\documentclass{article}

\pdfoutput=1

\usepackage{amsthm,stmaryrd}  
\usepackage{amsfonts, mathrsfs,amssymb}
\usepackage{amsmath,bbm}
\usepackage[numbers, sort&compress]{natbib}
\usepackage{graphicx}
\usepackage{enumerate}
\usepackage{paralist}
\usepackage{times}
\usepackage[labelfont=bf,font=it]{caption}
\usepackage[top=2cm, bottom= 2cm, left=2.3cm, right=2.3cm]{geometry}
\usepackage{wrapfig}

\newcommand{\cec}{\color{magenta}}

\newcommand{\fincec}{\color{black}}

\usepackage{pgf} 
\usepackage{tikz}
\usetikzlibrary{arrows,shapes}

\newtheorem{thm}{Theorem}[section]
\newtheorem{defn}[thm]{Definition}
\newtheorem{lem}[thm]{Lemma}
\newtheorem{prop}[thm]{Proposition}
\newtheorem{cor}[thm]{Corollary}
\theoremstyle{definition}

\usepackage{hyperref}

\allowdisplaybreaks

\newcommand{\cA}{\mathcal A}

\newcommand{\bq}{\mathbf q} 
\newcommand{\bs}{\mathbf s} 
\newcommand{\bds}[1]{\boldsymbol{#1}}
\newcommand{\sss}{\scriptscriptstyle}

\newcommand{\of}{\operatorname f}
\newcommand{\uni}{\text{uni}}
\newcommand{\bst}{\text{bst}}

\newcommand{\Lab}{\mathtt{Lab}}

\newcommand{\bff}{\mathbf f}

\newcommand{\bin}{\mathrm{Bin}}

\newcommand{\True}{\ensuremath{\mathrm{True}}}
\newcommand{\False}{\ensuremath{\mathrm{False}}}

\newcommand{\trim}{\mathtt{trim}}

\newcommand{\N}{\ensuremath{\mathbb N}}

\newcommand{\R}{\ensuremath{\mathbb{R}}}
\newcommand{\bbT}{\ensuremath{\mathbb{T}}}
\newcommand{\bbP}{\ensuremath{\mathbb{P}}}

\newcommand{\p}[1]{{\mathbb P}\left(#1\right)}
\newcommand{\pc}[1]{{\mathbb P}(#1)}

\newcommand{\Ec}[1]{\ensuremath{\mathbb{E} [#1]}}

\newcommand{\E}[1]{\ensuremath{\mathbb{E} \left[#1 \right]}}

\newcommand{\I}[1]{\ensuremath{\mathbf{1}_{ \{ #1 \} }}}

\newcommand{\ber}{\text{Ber}}

\hypersetup{
    bookmarks=true,         % show bookmarks bar?
    unicode=false,          % non-Latin characters in Acrobatâs bookmarks
    pdftoolbar=true,        % show Acrobatâs toolbar?
    pdfmenubar=true,        % show Acrobatâs menu?
    pdffitwindow=true,      % page fit to window when opened
    pdftitle={My title},    % title   
    pdfauthor={Author},     % author
    pdfsubject={Subject},   % subject of the document
    pdfnewwindow=true,      % links in new window
    pdfkeywords={keywords}, % list of keywords
    colorlinks=true,       % false: boxed links; true: colored links
    linkcolor=blue,          % color of internal links
    citecolor=blue,        % color of links to bibliography
    filecolor=blue,      % color of file links
    urlcolor=blue           % color of external links
}

\newcommand{\GW}{\ensuremath{\mathrm{GW}}}

\usepackage{nicefrac}

\usepackage{mathabx}
\def\EgalLoi{{~\mathop{= }\limits^{(law)}}~}
\newcommand{\Ess}{\mathtt{Ess}}

\newcommand{\Supp}{\mathrm{Supp}}

\begin{document}

%\begin{frontmatter}

\title{And/Or trees: A local limit point of view}
%\runtitle{And/Or trees: A local limit point of view}

\author{Nicolas Broutin\thanks{Inria Paris-Rocquencourt, Domaine de Voluceau, 78153 Le Chesnay, France and NYU Shanghai. \texttt{nicolas.broutin@inria.fr}.}~
and C\'ecile Mailler\thanks{Department of Mathematical Sciences, University of Bath, Claverton Down, BA2 7AY Bath, UK. \texttt{c.mailler@bath.ac.uk}. This author would like to thank EPSRC for support through the grant EP/K016075/1.}}
%\author{\fnms{Nicolas} \snm{Broutin}\ead[label=e1]{nicolas.broutin@inria.fr}}
%\address{Inria Paris-Rocquencourt,\\ Domaine de Voluceau,\\ 78153 Le Chesnay, France\\ \printead{e1}}
%\affiliation{Inria Paris-Rocquencourt and NYU Shanghai}
%\and
%\author{\fnms{C\'ecile} \snm{Mailler}\corref{}\ead[label=e2]{c.mailler@bath.ac.uk}\thanksref{t1}}
%\thankstext{t1}{This author would like to thank EPSRC for support through the grant EP/K016075/1.} 
%\address{Department of Mathematical Sciences,\\ University of Bath,\\ Claverton Down,\\ BA2 7AY Bath, UK\\ \printead{e2}}
%\affiliation{University of Bath}

%\runauthor{Nicolas Broutin and C\'ecile Mailler}

\maketitle

\begin{abstract}
We present here a new and universal approach for the study of random and/or trees,
unifying in one framework many different models, including some novel ones 
not yet understood in the literature.
An and/or tree is a Boolean expression represented in (one of) its tree shapes.
Fix an integer $k$, take a sequence of random (rooted) trees of increasing size, say
$(t_n)_{n\ge 1}$, and label each of these random trees uniformly at random in order 
to get a random Boolean expression on $k$ variables.

We prove that, under rather weak local conditions on the sequence of random trees $(t_n)_{n\ge 1}$, 
the distribution induced on Boolean functions by this procedure converges as $n$ tends to infinity. 
In particular, we characterise two different behaviours of this limit distribution 
depending on the shape of the local limit of $(t_n)_{n\ge 1}$: a \emph{degenerate} case when the 
local limit has no leaves; and a non-degenerate case, which we are able to describe in more details 
under stronger conditions. In this latter case, we provide a relationship between
the probability of a given Boolean function and its \emph{complexity}. 

The examples covered by this unified framework include
trees that interpolate between models with logarithmic typical distances 
(such as random binary search trees) and other ones with square root typical distances (such as 
conditioned Galton--Watson trees).
\end{abstract}


\section{Introduction}\label{sec:intro}
The problem of generating a \emph{complex} 
random Boolean function and understanding its typical properties 
can be traced back to the pioneering work of \citet{RiSh1942}, 
in which the authors studied uniformly random $k$-variable Boolean functions (for large integer~$k$). 
However, the uniform distribution is only natural if one represents the functions by a truth table 
(assigning a uniformly random value to every possible entry vector). 
Another way of representing a Boolean function is by Boolean \emph{expressions}, 
and significant efforts have been made towards defining probability 
distributions on Boolean functions via of their representation by Boolean expressions.
For a general introduction to such questions and related problems, we refer the reader to 
the survey article by \citet{Gardy2006a} and the references therein.

The concept of and/or trees arises as a natural representation of a Boolean expression, 
and the many different standard distributions of random trees can be used to sample random Boolean expressions.
More precisely, an expression is equivalent to 
a tree whose internal nodes are labelled 
by the Boolean connectors `and' ($\wedge$) or `or' ($\vee$), and whose leaves are 
labelled by literals among $x_1, \bar x_1, \dots, x_k, \bar x_k$, for some integer $k\ge 1$,
where $\bar x$ denotes the negation of $x$ (see Figure~\ref{fig:example_Btree}). 
We call such trees \emph{and/or trees}.
The origins of this line of research may be traced back to the works of \citet{Woods1997}. Amongst 
other things, he proved the existence of a limit probability distribution for Boolean function 
represented by sequences of trees of increasing sizes, and conjectured that there might be a 
relationship between the \emph{complexity} of a function and the probability that it is sampled.

%This is what motivated
%\citet{LeSa1997} to design an alternate model for which one could prove such a relation. 

% First: \cite{PaVeWi1994a} rescaled weight of the function (number of ones) and the measure 
% it induces on $[0,1]$.
% Distributions of formulas represented by trees for which the corresponding distribution 
% on Boolean functions converges to the uniform distribution \cite{Savicky1990,Savicky1995}
% Some specific model of tree with $\wedge$ and $\oplus$ \cite{Savicky1998}

\citet{Woods1997} considered functions represented by uniformly random Cayley trees (general rooted 
trees on $[n]=\{1,2,\dots, n\}$).
The case of Boolean functions encoded by uniformly random binary and/or trees, called the Catalan tree model, 
was first studied by \citet{LeSa1997}. Both papers proved the existence of 
a natural probability distribution on $k$-variable Boolean functions which is the 
weak limit of the probability induced by the uniform random trees on $[n]$ and
by the uniform binary and/or trees of size $n$, respectively. 
Since the results are similar, for the sake of presentation and only in this introduction, 
we focus on the case of binary trees in \cite{LeSa1997}: 
for a function $f$ in $k$ Boolean variables, the corresponding 
limit probability is denoted by $\mathbb P_k^{\uni}(f)$. 
\citet{LeSa1997} also obtain bounds on $\mathbb P_k^\uni(f)$ in terms of the complexity $L(f)$ of 
the function $f$ (being the minimal number of \emph{leaves} of a tree representing $f$) 
of the kind suggested by \citet{Woods1997}. More precisely, they prove that, for all $k\ge 2$,
\begin{equation}\label{eq:LeSa-bounds}
\exp(-c_1 L(f) \log k) \leq \mathbb P_k^\uni(f) \le \exp(- c_2 L(f) k^{-3}),
\end{equation}
for two constants $c_1,c_2>0$. 
The results in \cite{LeSa1997} were then reproduced and slightly improved by \citet{ChFlGaGi2004}
(who replace the $k^{-3}$ in the right-hand side of~\eqref{eq:LeSa-bounds} by $k^{-2}$).
These bounds are the only results in the literature that hold for fixed~$k$;
as a side remark, we will show in this article how to improve further  
the upper bound by a very simple symmetry argument.
This improvement is significant for two reasons:
the new upper bound constrains the probability of functions with complexity $o(k^2)$,
and we will show using a class of functions called {\it read-once} that it is sharp.

More precise bounds seem hard to obtain without considering the limit as $k$ tends to infinity. 
The first result in this direction was by \citet{Kozik2008a} 
who showed that for all integer $k_0$, for all $k_0$-variable Boolean function $f$,
%having a complexity independent of the number $k$ of variables,
\begin{equation}\label{eq:kozik-theta}
\mathbb P_k^\uni(f) = \Theta\left(\frac1{k^{L(f)+1}}\right)\text{ when }k\to+\infty,
\end{equation}
where the constants involved in the $\Theta$-term depend on~$f$.

More recently, it was proved by \citet{ChGaMa2011a} that if one replaces the 
uniformly random binary trees underlying the distribution defined in \citet{LeSa1997} by 
random binary search trees (see, e.g., \citet{Knuth1973b}), then the behaviour of the family of distributions 
induced on Boolean function is radically different: indeed, writing $\mathbb P_k^{\bst,n}$ for the probability 
induced on $k$-variable Boolean functions by trees of size~$n$, one has 
\[\mathbb P_k^{\bst,n}(\{\True, \False\}) \to 1, \text{ as }n\to+\infty,\]
for all integer $k$,
where $\True$ and $\False$ are the two constant functions:
$((x_1, \ldots, x_k)\mapsto \True)$ and $((x_1, \ldots, x_k)$ $\mapsto \False)$;
we say that this distribution is {\it degenerate}.

Although slight generalisations were considered in the literature (see for example \citet{GGKM15}),
essentially only the two models of random and/or trees described above 
(uniform binary tree and random binary search tree) have been studied.
However, there is no a priori reason why one should choose the underlying tree to 
be binary (as already mentioned in~\citet{GGKM15}, 
the conjunction and disjunction connectors are associative), 
nor any reason that justifies the uniform or the random binary search tree distributions,
apart of course from the fact that one can do some explicit computations in these cases. 
This is why we initiate in this paper 
a more general approach that is independent of the underlying family of trees.

Our aim is to place the previous studies in a common framework by introducing
a family of distributions on Boolean functions defined as weak limits 
of distributions coming from tree representations.
This family will include most of the and/or tree models studied in the literature, as well as
some models which behaviour was unknown up to now.
This paper is the first attempt at describing this family of probability distributions in greater generality.
In particular, by taking a local limit point of view, we generalize and greatly simplify the proofs of previously known results.
We extract essential properties that are needed to prove weak convergence of the probability distributions 
induced by a sequence of random trees, 
and to obtain bounds on the probability of a given Boolean function in terms of its complexity.

The main insight given by this common framework is that
the key property of the underlying sequence of random trees of increasing sizes 
is its local limit as opposed to its scaling limit: on the one hand,
a local limit with no leaf \emph{near the root} will induce a distribution on Boolean functions 
that is degenerate or concentrated on the two constant functions $\True$ and $\False$
(the case of random binary search tree of \cite{ChGaMa2011a} is a typical example);
on the other hand, a sequence locally converging to an infinite spine with reasonably small sub-trees 
hanging from it will verify Equation~\eqref{eq:kozik-theta}.
As corollaries, we obtain new proofs of the results of \citet{LeSa1997}, \citet{ChFlGaGi2004} and \citet{Kozik2008a} discussed above.

Furthermore, we enlarge the family of Galton--Watson trees for which one can obtain such results and also consider 
Ford's alpha model~\cite{Ford2005}, a model that can be seen as interpolating between the Galton--Watson family and the binary search tree, which has not up till now been considered in the context of and/or trees.

\section{Description of the framework and main results}

\subsection{Boolean trees and Boolean functions}\label{sub:def}
In this paper, for all integer $k$, a $k$-variable {\bf Boolean tree} (or an and/or tree), 
is a rooted finite tree having no node of arity one (i.e. having exactly one child), 
whose internal nodes are labelled by the disjunction connective $\lor$ or by the conjunction connective $\land$,
and whose leaves are labelled by literals taken in a finite fixed set $\{x_1, \bar x_1, \ldots, x_k, \bar x_k\}$,
where $\bar x_i$ stands for the negation of the variable $x_i$.

The set of and/or Boolean expressions can be defined recursively as follows:
a Boolean expression is either a literal or its negation,
or the conjunction of at least two Boolean expressions,
or the disjunction of at least two Boolean expressions.
Any Boolean tree is equivalent to a Boolean expression.% and thus represents a Boolean function.

Recall that for all integer $k$, a $k$-variable {\bf Boolean function} is a mapping from $\{\False,\True\}^{k}$ onto $\{\False, \True\}$.
For every assignment of the variables $(x_1, \ldots, x_k)\in\{\False, \True\}^k$, one can apply Boolean algebra rules from the leaves of a Boolean tree up to its root, and the value obtained at the root can be defined as the image of $(x_1, \ldots, x_k)$;
thus, every Boolean tree represents a Boolean function (but several trees can represent the same Boolean function as seen in Figure~\ref{fig:example_Btree}).
We denote by $\of[\tau]$ the Boolean function represented by the Boolean tree $\tau$.

The {\bf size} of a tree $t$, denoted by $\|t\|$, is the number of its leaves. 
The {\bf complexity} of a non-constant Boolean function $f$, denoted by $L(f)$, is the size of the smallest trees that represent it,
which we call {\bf minimal trees }of $f$ (see Figure~\ref{fig:example_Btree} for an example).
The complexity of the two constant Boolean functions, denoted respectively by $\True$ and $\False$, is~0.

\begin{figure}
\begin{center}
\includegraphics[width=0.6\textwidth]{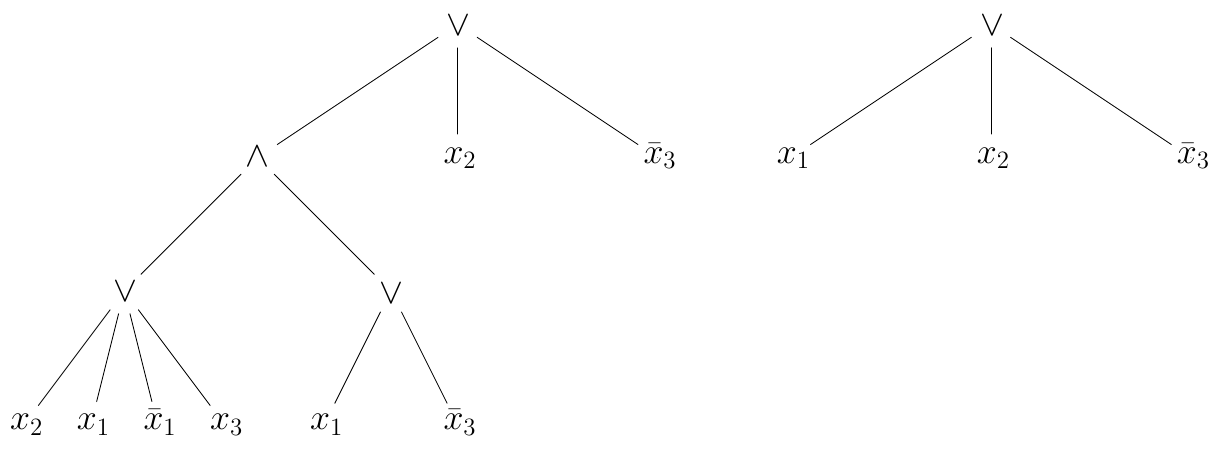}
\end{center}
\caption{Both this and/or trees represent the Boolean function $(x_1, x_2, x_3)\mapsto x_1 \lor x_2 \lor \bar x_3$. The left tree has size $8$. The right tree has size $3$ and is a minimal tree of the Boolean function $(x_1, x_2, x_3)\mapsto x_1 \lor x_2 \lor \bar x_3$, which thus has complexity~$3$.}
\label{fig:example_Btree}
\end{figure}

For any integer $k$, 
given a rooted tree $t$ having no node of arity one, 
we define its {\bf randomly $k$-labelled version} $\boldsymbol{\hat t}$ as follows 
(note that in order to keep notations simple the dependence in $k$ is not explicit in this notation):
each internal node of $t$ chooses a label in $\{\land, \lor\}$ uniformly at random,
each leaf of $t$ chooses a label in $\{x_1, \bar x_1, \ldots, x_k, \bar x_k\}$ uniformly at random,
independently from each other. 
For a Boolean function $g$, we denote by $P_k[t](g)$ the probability that 
the randomly $k$-labelled version of $t$ represents~$g$
(in other words, $P_k[t](g)$ is the probability that $\of[\hat t]$ equals $g$).

Except if mentioned otherwise, all trees considered in this article are assumed to contain no node of arity one.
This assumption is natural in the context of and/or trees since the two logical connectives $\land$ and $\lor$ are binary operators.
Note that thanks to associativity, they can also be considered as $r$-ary operators for all $r\geq 2$, 
which is why we do not restrict ourselves to binary trees, as sometimes done in the literature (see for example~\cite{LeSa1997, ChFlGaGi2004, Kozik2008a}).

\subsection{The local topology, infinite trees, and continuity results}
One of the main goals of the paper is to properly extend the above definition of a 
distribution $P_k[t]$ on the set of Boolean functions to (a certain class of) infinite trees.
Our approach relies on approximations of the infinite trees by sequences of growing trees
in the local topology around the root, which we now introduce.

For a (rooted) tree $t$ we define the truncation at height $h\ge 0$ as the 
subtree induced on the nodes at distance at most $h$ from the root, and denote it by $t^h$. 
A sequence of rooted trees $(t_n)_{n\ge 1}$ is said to {\bf converge locally} to a tree $t$ 
if for every integer $h\ge 1$ there exists $n_h$ large enough that for all $n\ge n_h$, 
the truncations $t_n^h$ are isomorphic to $t^h$. (Note that the limit tree $t$ does not 
have to be infinite.)
For a tree $t$, we define its number of {\bf ends} as 
the number of disjoint paths to infinity (more precisely, the limit of the number of 
connected components of the forest induced by $t$ on the set of nodes at distance at 
least $h\ge 1$, as $h\to \infty$; an end may also be defined as an equivalence class 
of infinite paths, where two paths are equivalent if their symmetric difference is finite).
When an infinite tree $t$ has a single end, the unique infinite simple path is called the 
{\bf spine}.

\vspace{\baselineskip}
The idea is to identify the distribution on Boolean functions encoded by an infinite tree 
as the limit distribution of the functions encoded by approximating sequences of growing trees. 
So given a sequence of trees $(t_n)_{n\geq 1}$ locally converging to a tree $t$, one is 
led to showing that the Boolean function $\of[\hat t_n]$ converges in distribution, as $n$ tends to infinity,
to a limit that only depends on $t$. The two following continuity theorems prove that this is the 
case when the limiting tree has no leaf or when it has finitely many ends.
Examples of such sequences of trees will be given later.

\vspace{\baselineskip}
\noindent\textsc{Trees without leaves.} Our first theorem deals with trees without leaves,
that yield degenerate distributions on Boolean functions:

\begin{thm}\label{thm:loc-tri}
(a) Suppose that $(t_{k,n})_{n\ge 1}$ is a sequence of trees converging locally to an infinite tree
without leaves. Then, as $n\to\infty$,
$$P_k[t_{k,n}](\True)=P_k[t_{k,n}](\False) \to \frac 1 2.$$ 
(b) Conversely, if $(t_{k,n})_{n\ge 1}$ converges locally to a tree with at least one leaf, then 
there exists a function $f\not \in \{\True,\False\}$ such that
$$\liminf_{n\to\infty} P_k[t_{k,n}](f)>0.$$
\end{thm}

%Note that we include the case when the sequence of trees $(t_{k,n})_{n\geq 1}$ depends on $k$: in this case,
%the limit tree $t$ a priori also depends on $k$, which is of course of no harm since $k$ is fixed.

Theorem~\ref{thm:loc-tri} implies the following straightforward corollary, which 
settles a conjecture of \citet{ChGaMa2015a}.
For $\sigma\in \N$, let $\bbT_\sigma$ denote the collection of rooted trees such that the 
leaf closest to the root lies at distance at least $\sigma$.

\begin{cor}\label{cor:random-loc-tri}
Let $(T_n)_{n\ge 1}$ be a sequence of random trees. Then
$$
\bbP(\of[\hat T_n]\in \{\True,\False\}) \to 1 
\quad \text{if and only if} \quad
\inf_{\sigma \ge 1} \lim_{n\to \infty}\bbP(T_n \in \bbT_{\sigma}) =1
$$
\end{cor}

Some special cases of Corollary~\ref{cor:random-loc-tri} have already been proved in the literature. 
If, for all $n\geq 0$, $T_n$ is almost surely equal to the balanced binary tree of height $n$ (i.e. the 
unique binary tree which has $2^n$ leaves, all lying at height $n$), then, it is proven by \citet{FoGaGe2009}
that $\bbP(\of[\hat T_n]=\{\True,\False\})\to~1$. 

The case when $T_n$ is a random binary search tree of 
size $n$, is treated by~\citet{ChGaMa2011a}.
Recall that a binary search tree of size $n$ is the rooted binary tree constructed as follows:
Given a list of distinct real numbers $\pi=(\pi_1,\pi_2,\dots, \pi_n)$, the root is labelled 
with $\pi_1$, and the tree is recursively obtained by repeating the construction with the lists contaning 
the $\pi_i$ that are smaller and larger than $\pi_1$ for the left and right subtrees of the root, 
respectively. The random binary search tree is then the binary search tree
obtained when $\pi$ is a sequence of i.i.d.\ random variables, for example uniformly distributed in $(0,1)$ 
(see e.g.~\cite{Knuth1973b}). 
\citet{ChGaMa2011a}'s result is a direct consequence of Corollary~\ref{cor:random-loc-tri} 
since \citet{Devroye1986} proved that the \emph{fill-up} or \emph{saturation level} 
(the height of the leaf closest to 
the root) of the random binary search tree of size $n$ is asymptotically equivalent to $c \log n$ 
in probability, where $c=.3733\dots$ is the unique solution smaller than one of $c\log((2e)/c)=1$
(see also \cite{Pittel1984}).
The same holds when the underlying tree shape is any of the classical random search trees 
based on the divide-and-conquer paradigm, for instance quad-trees or $k$-d trees built 
from uniformly random point sets, tries \cite{Mahmoud1992a}, or more generally any example 
that fits in the framework of \cite{BrDeMcSa2008}.

\vspace{\baselineskip}
\noindent\textsc{Trees with finitely many ends.}\ 
Our second continuity result concerns trees with finitely many ends for which 
the distribution on Boolean functions is non-degenerate.
The first theorem below ensures convergence to a limiting probability distribution when $n$ goes to infinity.
The fact that this distribution is non-degenerate is the next main result, discussed in Section~\ref{sub:discussion}.

% \sout{
% One of the most natural example of a sequence of random trees whose local limit does have leaves
% is the Catalan tree model: for all $n\geq 0$, the tree $T_n$ is a uniformly random binary tree of 
% size $n$. It is proven in~\cite{ChFlGaGi2004, Kozik2008a} that $f[\hat T_n]$ is not 
% asymptotically degenerate. This uniform case
% matches the hypothesis of our second continuity result:
% }
\begin{thm}\label{thm:loc-lim}
Fix $k\ge 1$. 
Suppose that $(t_{k,n})_{n\ge 1}$ is a sequence of trees which converges locally to $t_k$, and 
that $t_k$ has only finitely many ends. Then, there exists a probability distribution $P_k[t_k]$
such that for every Boolean function $f$ one has, as $n\to\infty$,
$$P_k[t_{k,n}](f) \to P_k[t_k](f).$$
\end{thm}

This convergence result is relatively easy to apply to a whole range of examples:
It is known since \citet{Grimmett1980} that the family trees of critical Galton--Watson trees
conditioned on the total progeny of increasing sizes converge locally to an infinite tree with a 
single end (see also \citet{AlSt2003}), so that the convergence results of \citet{LeSa1997} 
and \citet{ChFlGaGi2004} are straightforward consequences of Theorem~\ref{thm:loc-lim}.
We also provide in Section~\ref{sec:examples} novel examples of applications 
to random unordered trees (see \citet{MaMi2011a}), and other random trees arising from fragmentation 
processes (see \citet{HaMi2012a}).

\subsection{Trees with a unique end: properties of the limit distribution.}\label{sub:discussion}
In the non-degenerate case, we describe further the behaviour of the limit distribution.
For technical reasons, we restrict ourselves to random trees whose local limit has a unique end, and
under further (but reasonable) assumptions on the limit tree of the family $(T_n)_{n\ge 1}$, we are 
then able to prove the equivalent of~\citet{Kozik2008a}'s result (see Equation~\eqref{eq:kozik-theta}).
As in previous approaches by analytic combinatorics, the proof will take two steps: 
first estimate the probability of the two constant functions ($\True$ and $\False$) 
and then deduce from it the probability of a general Boolean function. 
More precisely, we are able to derive the asymptotic leading term of $P_k[T_n](f)$ when $k$ tends to infinity, 
for any $k_0$-variable Boolean function~$f$ (for any fixed integer $k_0$), 
in terms of its complexity.

Our main result in this direction needs further definitions before being properly stated, and we 
thus postpone its statement to Section~\ref{sec:bounds} (cf. Theorem~\ref{thm:theta}). However, 
it reads informally as follows: 
if the family $(T_n)_{n\geq 1}$ of random trees converges locally to an infinite 
spine on which are hung some i.i.d. forests, and if we can reasonably control the size of these forests,
then, for any integer $k_0$, for any $k_0$-variable Boolean function $f$, asymptotically as $k\to+\infty$,
\[\lim_{n\to+\infty} P_k[T_n](f) = \Theta\left(\frac1{k^{L(f)+1}}\right).\]
We apply this result to two main examples: 
critical Galton--Watson trees and Ford's alpha tree:
the following two theorems are proved in Sections~\ref{sub:ex_GW} and~\ref{sub:ex_alpha}.
\begin{thm}\label{prop:GW}
Let $(T_n)_{n\geq 1}$ be a critical Galton--Watson process of offspring distribution $\xi$ 
for which there exists a constant $a>0$ such that $\Ec{e^{a\xi}}<\infty$,
and such that $\mathbb P(\xi=1)=0$.
Then, for any integer $k_0$, for any $k_0$-variable Boolean function $f$,
asymptotically as $k\to+\infty$,
\[\lim_{n\to+\infty} P_k[T_n](f) = \Theta\left(\frac1{k^{L(f)+1}}\right).\]
\end{thm}
Note that the case of Catalan trees studied by Lefmann and Savick{\'y}~\cite{LeSa1997}, Chauvin et al.~\cite{ChFlGaGi2004} and Kozik~\cite{Kozik2008a}, is a particular case of Theorem~\ref{prop:GW}.

Ford's alpha model (see~\cite{Ford2005}) is defined as follows (cf. Figure~\ref{fig:alpha_tree}): 
$T_1^{\alpha}$ has a unique node to which is linked the root-edge. 
To build $T_{n+1}^{\alpha}$ from $T_n^{\alpha}$, 
weight each internal edge (the root-edge and any edge that is not linked to a leaf is an internal edge) by $\alpha$ 
and each external edge by $1-\alpha$. 
Pick at random an edge with probability proportional to these weights, 
add an internal node in the middle of this edge and link a new leaf to this new internal node. 
We call $T_n^{\alpha}$ the alpha tree of size $n$.
The model interpolates between the Catalan tree ($\alpha=\nicefrac12$) and the random binary search 
tree ($\alpha=0$).
It has been known since~\cite{HaMi2004a} that the expected height of the $n$-leaf alpha-tree behaves as $n^{\alpha}$, 
except for the extremal case $\alpha=0$ since the random binary search tree's height behaves as $\ln n$ (see~\cite{Devroye1986}).
This model thus permits to explore a whole range of {\it tree shapes}. 
The following theorem shows that only the case $\alpha=0$ corresponding to the random binary search tree 
induces a degenerate distribution on Boolean functions. 
Any other alpha tree verifies Equation~\eqref{eq:kozik-theta}:
\begin{thm}\label{prop:alpha_tree}
Let $(T_n^{\alpha})_{n\geq 1}$ be a sequence of alpha trees. 
Then, for all $\alpha\in(0,1]$,
for any integer $k_0$, for all $k_0$-variable Boolean function $f$,
as $k\to+\infty$,
\[\lim_{n\to+\infty} P_k[T^{\alpha}_n](f) = \Theta\left(\frac1{k^{L(f)+1}}\right).\]
\end{thm}
The behaviour of the distribution on Boolean functions that is induced by the alpha model 
was unknown up to now. Being able to prove Theorem~\ref{prop:alpha_tree} as a corollary 
of our main result is thus a significant step forward, when one compares it to the 
estimates that were available until now, that only concern very specific subcases. 

{\medskip \noindent\bf Remark:}
In Theorem~\ref{prop:GW} and~\ref{prop:alpha_tree}, as well as in Equation~\eqref{eq:kozik-theta},
the constants involved in the $\Theta$-term depend on the function~$f$.

\begin{figure}
\begin{center}
\includegraphics[width=.7\textwidth]{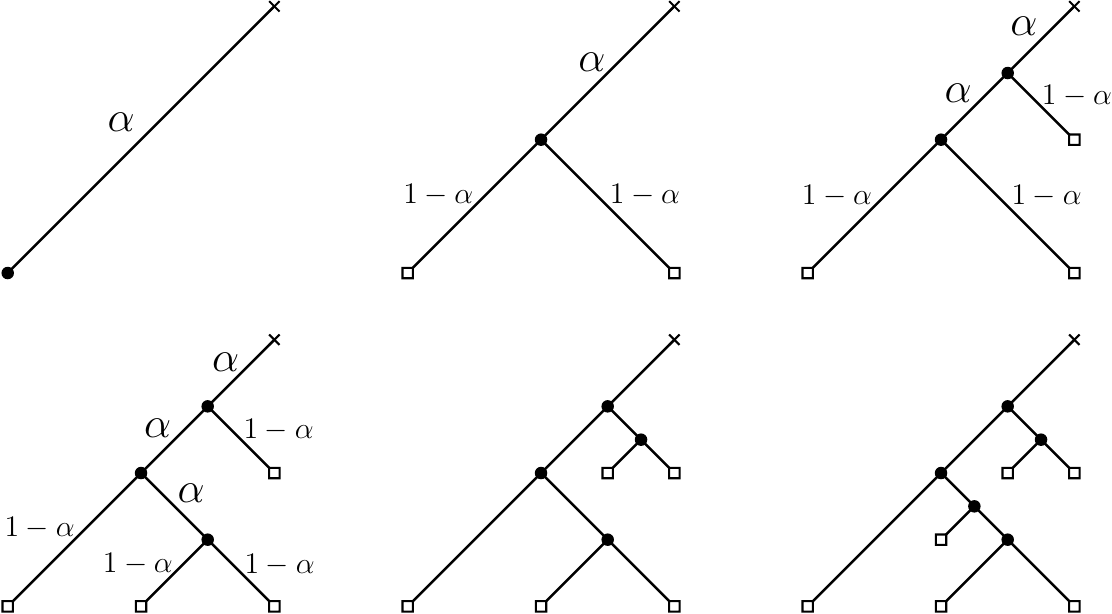}
\end{center}
\caption{A possible realisation of the first few steps of the construction of alpha trees.
The alpha tree itself is obtained when removing the edge attached to the crossed node (called the root-edge) 
and the cross node itself. The internal nodes are represented by disks and leaves by squares.}
\label{fig:alpha_tree}
\end{figure}

% {\red 
% \medskip
% \noindent\textsc{(No) Shannon effect.}\ 

% \cite{RiSh1942,Shannon1949a}
% \cite{GeGi2010a}

% }

\subsection{Discussion and remarks} 
It is important to note that Theorem~\ref{thm:loc-lim} is about deterministic sequences of trees. 
So, in the frameworks of the previous results where the trees are random, the limit distribution 
we construct is conditional on the limit tree. This type of results makes it hard to resort 
to arguments based on analytic combinatorics \cite{FlSe2009} and the Drmota--Lalley--Woods 
theorem giving asymptotics for generating functions satisfying a certain type of system of equations,
since these techniques are very intimately related to counting problems. The techniques we use here 
are probabilistic,
as in part of \citet{ChFlGaGi2004}. One of the drawbacks is that we cannot guarantee that the limit 
probability $P_k[T_{\infty}]$ charges every function in $k$ variables, though it should certainly be the case. 

We only consider the case of and/or trees, but one could of course think of 
other models of Boolean expressions. For instance, the case of Boolean expressions
encoded by trees labelled by implications connectors have been treated by \citet{FoGaGeGi2012a}. 

Put together, Theorems~\ref{thm:loc-tri} and~\ref{thm:loc-lim} already give 
a pretty good idea of the properties of random Boolean expressions obtained by 
labelling large trees uniformly and independently of the tree. It would be 
interesting to study what happens when one deviates from this setting. One can 
probably relax the condition of uniformity without much harm, but the 
dependence of the labelling and the tree seems to be a more challenging obstacle.
For example, our setting does not include non plane and/or trees as defined in~\citet{GGKM15}, 
a model that takes into account the commutativity of the conjunction and disjunction 
connectives; 
this model is not covered by our framework since it cannot be described 
as the random uniform labelling of a random tree.
 
In the introduction, we evoked two different representations of a Boolean function: 
the truth table and Boolean expressions seen as trees.
One could also think about Boolean functions that are not represented by trees but by
circuits modelled by directed acyclic graphs with a single sink.

Finally, it would be interesting to look more precisely at what happens when the 
underlying limit tree has infinitely many ends as well as leaves. Note for example that 
Theorem~\ref{thm:loc-tri} does not rule out the possibility that the limit 
Boolean function be a measurable function of the labelled limit tree (even when it 
has no leaves). We strongly believe that if one considers the growth of the number of
ends which intersect the ball of radius $d$ around the root (in the graph distance) as fixed, then for 
every small enough number of variables one should be able to define the limit Boolean function. 
Can one make such a claim more precise using for instance the branching 
number or the malthusian parameter in the case of Galton--Watson trees 
(as in \citet{BaPePe2006a})?

\vspace{\baselineskip}
\noindent\textbf{Plan of the paper.}\ 
Section~\ref{sec:local-limit} is devoted to the convergence to a limit probability distribution
as $n$ tends to infinity: it contains the proofs of Theorems~\ref{thm:loc-tri} and~\ref{thm:loc-lim} 
and provides some examples of families of trees for which these theorems apply. 
% In this section, we only assume convergence of the sequence $(T_n)_{n\ge 1}$ 
% to a limit tree that has exactly one infinite spine. 
% The case of finitely-many spines is then easily deduced.
Section~\ref{sec:improved_bounds} is focusing on the Catalan tree case: 
we present simple arguments to tighten Inequality~\eqref{eq:LeSa-bounds}.
In Section~\ref{sec:bounds} we prove the analog of~\citet{Kozik2008a}'s result 
(Equation~\eqref{eq:kozik-theta}) in our general setting.
This stronger result only holds under further moment assumptions which we discuss by providing examples.

\section{Continuity in the local topology}
\label{sec:local-limit}

\subsection{The degenerate case: proof of Theorem~\ref{thm:loc-tri}}
Note that local convergence to an infinite tree with no leaves 
is equivalent to the divergence of the saturation level (being the height of the closest leaf to the root), 
and in this section we phrase Theorem~\ref{thm:loc-tri} in this framework. 
For an integer $\sigma \ge 0$, we denote by $\bbT_{\sigma}$ the set of all rooted trees 
with a saturation level at least $\sigma$, so in particular $\bbT_0$ is the set of all rooted trees. 
Our proof of Theorem~\ref{thm:loc-tri} consists in estimating the probability that 
the random Boolean functions assigns two different values to two (distinct) points 
$a, b \in \{0,1\}^k$. (This is already the approach in~\cite{FoGaGe2009} and~\cite{ChGaMa2011a}.)
Note that we use the canonical notation $0=\False$ and $1=\True$.

Let $a$ and $b$ be two distinct elements of $\{0,1\}^k$, 
and let $\alpha,\beta\in\{0,1\}$. 
Let us define the following probability (where the probability $\mathbb P$ refers to the uniform random $k$-labelling): 
for all trees~$t$,
\[\bbP^{\alpha\beta}_t(a,b) := \mathbb P(\of[\hat t](a)=\alpha \text{ and }\of[\hat t](b)=\beta),\]
and the following supremum:
\[S_{\sigma}^{\alpha\beta}(a,b) = \sup\{\bbP^{\alpha,\beta}_t(a,b): t\in\bbT_{\sigma}\}.\]
% Instead of considering a deterministic tree $t$, 
% we want to apply the preceeding to a random sequence of trees $(T_m)_{m\geq 0}$. 
% Therefore, it will induce a double notion of randomness in our model: 
% the trees $(T_m)_{m\geq 0}$ are random variable and we label them randomly 
% in order to study the random variables $(f(\hat{T}_m))$. 
% Our aim is to prove the following theorem:
% \begin{thm}
% \label{th:th}
% Let $(T_m)_{m\geq 0}$ be a sequence of random (non labelled) trees. 
% We have the following limiting behaviour:
% \begin{equation}
% \label{eq:degen}
% \mathbb{P}(f(\hat{T}_m) = \True) = \mathbb{P}(f(\hat{T}_m) = \False)\to\frac12\text{ when }m\to\infty,
% \end{equation}
% if and only if the saturation level $s_m$ of $T_m$ verifies
% \[\lim_{m\to\infty} s_m = +\infty\text{ in probability}.\]
% \end{thm}

% The proof of this theorem is divided in two parts (Propositions~\ref{prop:prop1} and~\ref{prop:prop2}) 
% which correspond to the two implications contained in the equivalence announced in Theorem~\ref{th:th}. 
% The lemmas only consider the randomness of the labelling while the Propositions take into account 
% both the rendomness of the labellings and of the shapes $(T_n)_{n\geq 0}$.

\begin{lem}\label{lem:pre}
For every $k\ge 1$, 
there exists a constant $\sigma_0\ge 1$ such that for all $\sigma\ge \sigma_0$, one has
\[ \sup_{a\ne b \in \{0,1\}^k}S_{\sigma}^{10}(a,b) 
= \sup_{a\ne b \in \{0,1\}^k}S_{\sigma}^{01}(a,b) \le \frac 2 \sigma.
\]
\end{lem}

\begin{proof}
Fix any two distinct points $a,b\in \{0,1\}^k$, and, if there is no possible confusion, 
write $\bbP^{\alpha\beta}_t$ instead of $\bbP^{\alpha\beta}_t(a,b)$ 
and $S_{\sigma}^{\alpha\beta}$ instead of $S_{\sigma}^{\alpha\beta}(a,b)$. 
First of all, the symmetries of the labelling imply immediately that 
$\bbP^{10}_t = \mathbb{P}^{01}_t$ and $\bbP^{11}_t = \bbP^{00}_t$. 
Indeed, since the probabilities of $\land$ and $\lor$ are equal, 
and the probabilities of a variable $x_i$ and its negation $\bar x_i$ are also equal, 
we can conclude that for any finite tree $t$ and for any Boolean function $f$
and its negation $\neg f$ we have
\[P_k[t](f) = P_k[t](\neg f).\]
Moreover, for every tree $t$, we have $\bbP_t^{10}+\bbP_t^{11} = \frac12$,
and thus, $S_{\sigma}^{\alpha\beta}\leq \frac12$, for all $\alpha, \beta\in\{0,1\}$ and for all $\sigma\in\mathbb N$.

Let $t\in \bbT_{\sigma}$. 
Let $r\geq 2$ be the degree of the root of $t$ 
(recall that, by assumption, $t$ contains no node of arity one, see Section~\ref{sub:def})
and denote by $t_1,\ldots, t_r$ 
the subtrees rooted at the root's children.
For all $i\in\{1,\ldots, r\}$, $t_i\in\bbT_{\sigma-1}$. 
Moreover, $\of[\hat t](a)=1$ and $\of[\hat t](b)=0$ imply that
\begin{itemize}
	\item either the root is labelled by $\land$ and $\of[\hat t_i](a)=1$ for all $1\le i\le r$, and 
	at least some $\of[\hat t_i](b)=0$, 
	\item or the root is labelled by $\lor$ and $\of[\hat t_i](b)=0$ for all $1\le i\le r$, and 
	at least some $\of[\hat t_i](a)=1$. 
\end{itemize}
It follows that
\begin{align*}
\bbP_t^{10} 
&= \frac12 \left(
\prod_{i=1}^{r}(\bbP_{t_i}^{10} + \bbP_{t_i}^{11})
- \prod_{i=1}^r \bbP_{t_i}^{11} \right)
+ \frac 1 2 \left(\prod_{i=1}^{r}(\bbP_{t_i}^{10}+\bbP_{t_i}^{00})
- \prod_{i=1}^r \bbP_{t_i}^{00}
\right)\\
&= \frac1{2^r} - \prod_{i=1}^r \mathbb{P}_{t_i}^{11}
= \frac1{2^r} - \prod_{i=1}^r \left(\frac12 - \mathbb{P}_{t_i}^{10}\right).
\end{align*}
Now, for all $i\in\{1,\ldots, r\}$, we have $\mathbb{P}^{10}_{t_i}\leq S_{\sigma-1}^{10}$, 
which implies that, for every tree $t\in\bbT_{\sigma}$ with a root of degree~$r$,
\[\bbP_t^{10}\leq \frac1{2^r} - \left(\frac12 - S_{\sigma-1}^{10}\right)^r.\]
Thus,
\[S_{\sigma}^{10} 
= \sup_{t\in\bbT_{\sigma}} \bbP_t^{10}
\leq \sup_{r\geq 2}\left\{\frac1{2^r} - \left(\frac12 - S_{\sigma-1}^{10}\right)^r\right\}.
\]
One easily verifies that, for all $x\in [0,\nicefrac12]$, one has
\[\sup_{r\geq 2}\left\{\frac1{2^r} - \left(\frac12 - x\right)^r\right\}
= \frac1{4} - \left(\frac12 - x\right)^2,\]
which implies, since $S_{\sigma}^{10}\leq \frac12$ for all integer $\sigma$,
\[S_{\sigma}^{10} 
\leq \frac14 - \left(\frac12 - S_{\sigma-1}^{10}\right)^2.
\]
Let us define the sequence $(u_{\sigma})_{\sigma \geq 1}$ as follows: 
$u_1 = S_{1}^{10}$ and $u_{\sigma+1}= g(u_{\sigma})$, 
where $g(x) = \nicefrac14 - (\nicefrac12 - x)^2 = x-x^2$. 
Since $S_\sigma^{10}\le \nicefrac12$ for every $\sigma\ge 1$, a straightforward induction 
on $\sigma$ shows that, for all $\sigma\geq 1$, one then has $u_{\sigma} \geq S_{\sigma}^{10}$.

Note that $0$ is the unique fixed point of $g$ and that $[0,1]$ is stable by $g$. 
Moreover, $g'(0) = 1$ and $g''(0) = -2$. 
By a standard result about inductive sequences: 
as $\sigma\to +\infty$,
\[u_{\sigma} \sim -\frac{2}{\sigma g''(0)} = \frac1{\sigma}.\]
It follows that, for all $a\neq b$, there exists $\sigma_{a,b}$ such that, for all $\sigma \geq \sigma_{a,b}$,
$S_{\sigma}^{10}(a,b) \leq \nicefrac{2}{\sigma}$.
Let $\sigma_0 = \max_{a\neq b \in\{0,1\}^k} \sigma_{a,b}$ to conclude the proof.
\end{proof}

With Lemma~\ref{lem:pre} under our belt, the proof of Theorem~\ref{thm:loc-tri} (a) 
appears now as a straightforward application of the union bound.

\begin{proof}[Proof of Theorem~\ref{thm:loc-tri} (a)] 
Fix an arbitrary integer $\sigma\ge 1$. Then, by assumption, there exists $n_0\ge 1$ large enough
such that for all $n\ge n_0$, one has $t_n\in \bbT_\sigma$. 
% By assumption, for all $\sigma>0$,
% \[\mathbb{P}(s_m\geq \sigma)\to 1\text{ when }m \to+\infty.\]
% It exactly means that
% \[\mathbb{P}(T_m\notin\mathcal{T}_{\sigma})\to 0\text{ when }m \to+\infty.\]
Now, by definition of $S_\sigma^{10}$, for all $n\ge n_0=n_0(\sigma)$ we have
\begin{align*}
\bbP(\of[\hat t_n]\notin \{\True,\False\}) 
&= \mathbb{P}(\exists a,b\in\{0,1\}^k : \of[\hat t_n](a)\neq \of[\hat t_n](b))\\
&\leq \sum_{a\neq b} \bbP(\of[\hat t_n](a)\neq \of[\hat t_n](b))\\
&\le 2^{2k} \sup_{a\ne b} 2 S_\sigma^{10}(a,b).
\end{align*}
Now for any $\sigma\ge \sigma_0$ in Lemma~\ref{lem:pre}, and all $n\ge n_0(\sigma)$, we have
$$\bbP(\of[\hat t_n]\notin \{\True,\False\}) \le \frac{2^{2k+2}} \sigma.$$
Letting $\sigma\to\infty$ completes the proof. 
\end{proof}

We now move on to the lower bound in Theorem~\ref{thm:loc-tri}. 
% To state Theorem~\ref{th:th}, we thus only have to prove that if the saturation level of the sequence of random trees $(T_n)_{n\geq 0}$ 
% does not tend to infinity in probability, when $n$ tends to infinity, then, 
% the induced probability on $\mathcal{F}_n$ is not degenerate.

\begin{lem}\label{lem:lem2}
Let $\mathcal{A}_{\sigma}$ be the set of trees with saturation level equal to $\sigma$. 
Then, for every integer $\sigma\ge 1$,
\[I_{\sigma} := \inf_{t\in \mathcal{A}_{\sigma}} \bbP(\of[\hat t] \notin \{\True,\False\})
\ge 
\frac1{2^{\sigma}}.\]
\end{lem}

\begin{proof}
Let $t$ be a tree in $\mathcal{A}_{\sigma}$ and consider the associated randomly labelled tree $\hat{t}$. 
Let us denote by $x$ the label of $\ell$, one of its leaves at height $\sigma$; 
by $\diamond_0,\ldots,\diamond_{\sigma-1}$ the connectives of the nodes between the root and the leaf $\ell$ 
($\diamond_0$ being the label of the root and $\diamond_{\sigma-1}$ the label of the parent of $\ell$), 
and by $g_0,\ldots,g_{\sigma-1}$ the random boolean functions calculated by the forests hanging along the path from the root to $\ell$ 
(from top to bottom). Therefore, the function calculated by $\hat{t}$ is given by
\[\of[\hat{t}] = (g_0 \diamond_1 (g_2 \diamond_2 (\ldots (g_{\sigma-1}\diamond_{\sigma-1} x)))).\]
Let $j$ be the minimum $m\in\{0,\ldots,\sigma-1\}$ such that $g_m$ depends on $x$; 
if such~$m$ does not exist, we let $j=\sigma$. 
Let $\rho = g_{j+1}\diamond_{j+1}(\ldots (g_{\sigma-1}\diamond_{\sigma-1} x))$.
We can choose $\diamond_{j}$ in $\{\wedge, \vee\}$ such that $g_{j}\diamond_j \rho$ does depend on $x$: 
if $g_{j} \neq \False$, then $x \land g_{j}$ does depend on $x$ and if $g_{j}= \False$, 
then $x \lor g_{j} = x$ does depend on $x$. 
By induction, we then can choose $\diamond_{j},\ldots,\diamond_{0}$ such that 
$\of[\hat t]$ does depend on $x$, 
and the probability, conditionally on all the rest of the tree (and for any such conditioning), 
that $\diamond_0,\ldots,\diamond_{j}$ are actually equal to this choice is equal to $2^{-(j+1)}$.

To conclude, if we denote by $\of[\hat{t}](x=1)$ (resp. $\of[\hat{t}](x=0)$) 
the restriction of $\of[\hat{t}]$ to the subset of $\{0,1\}^n$ where $x=1$ (resp. $x=0$), 
\[\bbP(\of[\hat{t}](x = 1) \neq \of[\hat{t}](x = 0)) \geq \frac{1}{2^{{\sigma}}},\]
and thus
\[\mathbb{P}(\of[\hat{t}] \neq \True \text{ and }\of[\hat{t}]\neq \False) \geq \frac{1}{2^{{\sigma}}}.\]
The last inequality holds for every $t\in\mathcal{A}_{\sigma}$: 
taking the infimum proves Lemma~\ref{lem:lem2}.
\end{proof}

% \begin{prop}
% \label{prop:prop2}
% Let us assume that $s_m$ (the saturation level of $T_m$, cf. Theorem~\ref{th:th}) does not tend to $+\infty$ in probability.
% Then, there exists $\alpha>0$ such that, for all $m_0\geq 0$, there exists $m\geq m_0$ such that,
% \[\mathbb{P}(f(\hat{T}_m) = \True) = \mathbb{P}(f(\hat{T}_m) = \False) < \frac12-\alpha.\]
% \end{prop}

\begin{proof}[Proof of Theorem~\ref{thm:loc-tri} (b)]
By assumption, the saturation level of $(t_n)_{n\ge 1}$ does not diverge, and 
there exists $\sigma\ge 1$ such that we can find an infinite subsequence $J\subseteq \N$ such 
that for all $n\in J$, $t_n\not\in \bbT_\sigma$. Without loss of generality, we suppose now
that $J=\N$. 
% By assumption, there exists $\sigma>0$, there exists $\varepsilon>0$ such that, for all $m_0\in\mathbb{N}$, there exists $m\geq m_0$ such that
% \[\mathbb{P}(s_m\geq \sigma)<1-\varepsilon.\]
Then, for all $n\in \N$, 
\begin{align*}
\bbP(\of[\hat t_n]\notin\{\True,\False\}) 
%&\geq \bbP(f(T_n)\notin\{\True,\False\} \text{ and } s_m < \sigma)\\
% &= \sum_{k=0}^{\sigma-1} \mathbb{P}(f(\hat{T}_m)\notin\{\True,\False\} \text{ and }s_m = k)\\
&= \sum_{m=0}^{\sigma-1} \sum_{t\in \cA_m} \bbP(\of[\hat t_n]\notin\{\True,\False\}) \I{t_n = t}\\
% \end{align*}
% By independence between the random tree $T_m$ and its random labelling, we get,
% \begin{align*}
% \mathbb{P}(f(\hat{T}_m)\notin\{\True,\False\}) 
% &= \sum_{k=0}^{\sigma-1}\sum_{t\in \mathcal{A}_k} \mathbb{P}(f(\hat{t})\notin\{\True,\False\})\mathbb{P}(T_m = t)\\
&\geq \sum_{m=0}^{\sigma-1}\sum_{t\in \mathcal{A}_m} I_m \I{t_n = t},
\end{align*}
since $I_m = \inf_{t\in\cA_m} \bbP(\of[\hat t]\notin \{\True,\False\})$. 
Therefore, since $t_n\in \cA_m$ for some $m<\sigma$, we have by Lemma~\ref{lem:lem2}
\begin{align*}
\bbP(\of[\hat t_n]\notin\{\True,\False\})
&\geq \sum_{m=0}^{\sigma-1} \frac1{2^m} \sum_{t \in \cA_m}\I{t_n = t}\geq \frac1{2^{\sigma -1}},
\end{align*}
which proves the claim.
% Defining $\alpha = \frac{\varepsilon}{2^{\sigma}}$, we prove Proposition~\ref{prop:prop2} and thus achieve the proof of Theorem~\ref{th:th}.
\end{proof}

\subsection{Finitely many ends: proof of Theorem~\ref{thm:loc-lim}}
\label{sec:convergence}
In this section, we prove Theorem~\ref{thm:loc-lim}. 
So far, the Boolean function associated to a labelled tree has only been defined for finite trees. 
One of the main ingredients of the proof of Theorem~\ref{thm:loc-lim} is the following lemma, 
which proves that the Boolean function $\of[\hat t]$ is also well-defined when $t$ has a unique 
infinite path, which we refer to as the \emph{spine}.

Given a tree $t$ and an integer $h$, we denote by $t^h$ the tree obtained from $t$ 
by removing all nodes having height greater than $h$. 

\begin{lem}\label{lem:one_spine}
Let $t$ be a locally finite tree with at most one end. 
Then there exists an (random) integer $h\geq 1$ such that for all $i\geq h$, 
$\of[\hat{t}^i]=\of[\hat{t}^h]$ almost surely. The Boolean function $\of[\hat t]$ is then 
defined as $\of[\hat t^h]$. 
\end{lem}

\begin{proof}
Let us denote by $(u_i)_{i\ge 0}$ the sequence of nodes along the spine of $t$ (starting from the root), 
and write $\diamond_i$ the label of $u_i$ in $\hat{t}$. 
For convenience, we introduce another truncation of the infinite tree $t$: we let 
$t^{[h]}$ denote the subtree of $t$ containing the root when the spine is cut between the nodes
$u_h$ and $u_{h+1}$; then $t^{[h]}$ is finite for every $h\ge 0$. 
% By assumption, removing all the edges $\{u_i,u_{i+1}\}$, $i\ge 1$, on the spine from $t$ 
% yields a forest of finite trees $(t_i)_{i\ge 1}$, each one rooted at some node in $(u_i)_{i\ge 1}$. 
Let $(t_{i,j})_{j\ge 1}$ denote the sequence of finite subtrees rooted at the children of $u_i$ 
in an arbitrary order (that is, we except the tree rooted at $u_{i+1}$). 
Note that by assumption, for every $i\ge 1$, $(t_{i,j})_{j\ge 1}$ is actually a non-empty 
(because $t$ contains no unary nodes by assumption) finite sequence 
of finite trees. 
Let $f_{i,j}:=\of[\hat t_{i,j}]$ (the Boolean function represented by the random labelling of the tree $t_{i,j}$),
and note that $(f_{i,j})_{i,j\ge 1}$ is independent of the 
sequence of labels along the spine $(\diamond_i)_{i\ge 1}$. 
Say that two sequences $\bff_i:=(f_{i,j})_{j\ge 1}$ and $\bff_m:=(f_{m,j})_{j\ge 1}$ are equivalent if 
the collections of Boolean functions they are made of are identical (the multiplicities and ordering 
may be different); we then write $\bff_i\sim \bff_m$.

The proof consists in finding an integer $h\ge 1$ such that replacing the subtree of $t$ rooted 
at $u_h$ by \emph{any} other tree, finite or not, does not affect the Boolean function computed 
by the truncations $\hat t^{[i]}$, for all $i\ge h$. It is then possible to safely define $\of[\hat t]$ as 
$\of[\hat t^{[h]}]$.
Let $h(\hat t)$ be the minimal height $h\ge 0$ such that for all $i\ge h$ one has 
$\of[\hat t^{[i]}]=\of[\hat t^{[h]}]$. 
In order to complete the proof, it suffices to show that $h(\hat t)$ is almost surely finite.

Assume that there exists two integers $i<j$ such that $\diamond_i=\land$, $\diamond_j =\lor$, and $\bff_i\sim \bff_j$.
Then, there exists $m\geq 1$ such that $f_{j,m} = f_{i,1}$.
For all $x\in\{0,1\}^k$ such that $f_{i, 1}(x) = 0$, we have $\of[\hat t^{[m]}](x) = \of[\hat t^{[i]}](x)$ for all $m\geq i$, 
and for all $x\in\{0,1\}^k$ such that $f_{i, 1}(x) = f_{j, m}(x) = 1$, we have $\of[\hat t^{[m]}](x) = \of[\hat t^{[j]}](x)$ for all $m\geq j$.
Thus, for all $x\in\{0,1\}^k$, for all $h\geq j$, $\of[\hat t^{[m]}](x) = \of[\hat t^{[j]}](x)$.
The same result holds if $\diamond_i=\lor$ and $\diamond_j =\land$;
we define $\blackdiamond_i = \land$ (resp. $\lor$) if $\diamond_i=\lor$ (resp. $\land$).
We have shown that if there exist two integers $i<j$ such that
$\diamond_j = \blackdiamond_i$ and $\bff_i\sim \bff_j$,
then $\of[\hat t^{[j]}]=\of[\hat t^{[m]}]$ for every $m\ge j$.
If this occurs, $h(\hat t)\le j$, and we are now looking for such a pair $(i,j)$ of integers, 
that we call \emph{good} hereafter.

Invoking the pigeon hole principle, we know that among the $2^{2^{2^k}}+1=: s_k$ first sequences in 
$(\bff_i)_{i\ge 1}$ at least two are equivalent: almost surely, there exists 
$i<j\le s_k$ such that  $\bff_i\sim\bff_j$. 
Thus, since $(\bff_i)_{i\ge 1}$ is independent of $(\diamond_i)_{i\ge 1}$ and the latter is 
an i.i.d.\ sequence, with probability $\frac12$\, we have $\diamond_j=\blackdiamond_i$ and 
$(i,j)$ is a good pair. 
If $\diamond_j\neq \blackdiamond_i$, we look at the next $s_k$ 
sequences of $(\bff_i)_{i\geq 1}$, find two equivalent sequences by the pigeon hole principle, 
and the corresponding indices happen to be a good pair with probability $\frac 12$. 
Continuing in this manner, we see that the number of groups of size $s_k$ one has to 
look at before finding a good pair follows a geometric distribution of parameter $1/2$, 
so that $h(\hat t)$ is almost surely finite and the proof is complete.
\end{proof}

This proof contains a first cutting algorithm of an infinite randomly labelled Boolean tree $\hat t$, 
which is far from optimal, but sufficient to prove the continuity of $\of[\cdot]$ in the local topology.
The infinite Boolean tree can certainly be simplified further and we introduce a refined trimming 
algorithm later on.

If, instead of having a single end, the tree $t$ has finitely many ends,
Lemma~\ref{lem:one_spine} still holds and the proof remains the same: one has
to continue the cutting algorithm as long as an end remains. 
Or equivalently, one just uses the previous algorithm for the portions of he tree below height~$d$, 
where the ends have all been separated.  
This is straightforward and we omit a formal proof of the following lemma:

\begin{lem}\label{lem:sev_spine}
Let $t$ be a tree with finitely many ends. 
Then there exists a random integer $h\geq 1$ such that for all $i\geq h$, 
$\of[\hat{t}^i]=\of[\hat{t}^h]$ almost surely. The Boolean function $\of[\hat t]$ is then 
defined as $\of[\hat t^h]$. 
\end{lem}

\begin{lem}\label{lem:loc_conv}
Let $(t_n)_{n\geq 1}$ a sequence of unlabelled trees converging locally to a tree $t_{\infty}$ 
having finitely many spines. Then, there exists (random) integers $h\geq 1$ and $n_0\ge 1$
such that for all $i\geq h$, and all $n\geq n_0$, 
$$\of[\hat t_n^i]\EgalLoi \of[\hat t_{\infty}^i] = \of[\hat t_\infty].$$
\end{lem}

\begin{proof}
Lemma~\ref{lem:sev_spine} tells us that there exists almost surely an integer $h\geq 1$ such that, 
for all $k\geq h$, $\of[\hat t_{\infty}^k]=\of[\hat t^h_{\infty}]$, the latter serving as a definition 
for $\of[\hat t_\infty]$. 
On the other hand, for this value of $h$, since $t_n\to t$ in the local topology,
there exists a integer $n_0\geq 1$ such that, for all $n\geq n_0$, $t_n^h=t_{\infty}^h$, 
which implies that $\of[\hat t_n^h]$ and $\of[\hat t_{\infty}^h]$ have the same distribution. 

In other words, there exists almost surely an integer $h\geq 1$ such that, 
for all $i\geq h$, there exists an integer $n_i\geq 1$ such that, 
for all $n\geq n_i$,
\[\of[\hat t_n^i] \EgalLoi \of[\hat t_{\infty}^i] = \of[\hat t_{\infty}],\]
which proves the claim.
\end{proof}

The following result is a direct consequence of Lemma~\ref{lem:loc_conv}:
\begin{thm}\label{thm:cv}
Let $(T_{k,n})_{n\geq 1}$ a sequence of unlabelled random trees converging in distribution 
to a local limit $T_{\infty}$ having a finite number of infinite branches with probablity one. 
For any $k$-variable Boolean function~$f$, 
let us denote by $\mathbb{P}_{n,k}(f)$ the probability that $\of[\hat{T}_n]=f$. 
Asymptotically as $n$ tends to infinity, 
the sequence $(\mathbb{P}_{n,k})_{n\geq 1}$ converges to an asymptotic probability distribution $\mathbb{P}_k$ 
such that, for every Boolean function $f$, $\mathbb{P}_k(f) = \mathbb{P}(\of[\hat{T}_{\infty}]=f)$.
\end{thm}

Most natural examples have finitely many ends, but it is natural to ask whether this condition is necessary. 
Theorem~\ref{thm:loc-tri} treats the case when the limit tree has no leaves and
possibly infinitely many ends. 
One could ask what could be said in a less extreme case, when there are infinitely many ends
as well as leaves in the limit tree. This question remains open.

\subsection{A few natural examples}
\label{sec:examples}

To prove the existence of a limit probability distribution, it is enough to prove 
the local weak convergence towards a limit tree that has finitely many ends. 
%Both are  usually easy to verify and we now provide a few examples. 
We provide in this section examples of sequences of random trees that have finitely many ends, 
and thus to which Theorem~\ref{thm:loc-lim} can be applied.
In three of these examples, namely the conditioned critical Galton--Watson trees, the non-plane binary trees, 
and the fragmentation trees,
the sequence of trees $(t_{n,k})_{n\geq 0}$ and its local limit $t_k$ 
do not depend on $k$, the number of variables. 
For the so-called {\it associative tree}, treated in Section~\ref{sub:assoc}, which is a natural example from the literature, the 
sequence of trees and its limit both depend on the number of variables.

\subsubsection{Conditioned critical Galton--Watson trees}\label{sub:GW}
Let $\xi$ be an integer-valued random variable.
The Galton--Watson tree of progeny distribution $\xi$ is a random rooted tree in which 
every node has a number of children that is an independent copy of $\xi$.
% \begin{itemize}
% \item the root of the Galton--Watson tree has $\xi$ children;
% \item we then treat the new nodes in breadth first search order until every node has been treated (note that this procedure may be infinite);
% \item each node of the Galton--Watson tree samples its number of children according the the law of $\xi$, independently from each other.
% \end{itemize}

Assume that the progeny distribution verifies that $\E{\xi}=1$, and that $p_1=\pc{\xi=1}=0$ 
(recall that all trees considered in this article have no unary nodes).
We let $\GW_\xi$ denote the distribution of a Galton--Watson tree with reproduction
distribution $\xi$.
Let $T_n$ be a random Galton--Watson tree with progeny distribution $\xi$, 
conditioned on the total population being $n$ (if such a size is possible). Then, it 
is well-known that $T_n$ converges locally in distribution to an infinite tree 
$T_\infty$ described as follows (see, e.g., \cite{Kesten1986a, LyPePe95a, AbDe2014a}). 
Let $\hat \xi$ be the size-biased distribution associated to $\xi$ defined by
$\pc{\hat\xi=i}= i \p{\xi=i}$, $i\ge 0$. 
Let $(\hat \xi_i)_{i\ge 0}$ be a sequence of i.i.d.\ copies of $\hat \xi$. 
Then there exists a unique self-avoiding infinite path in $T_\infty$, 
consisting of the nodes $(u_i)_{i\ge 0}$; and for every $i\ge 0$, 
$u_i$ has $\hat \xi_i\ge 1$ children, one of each is $u_{i+1}$ and the others 
are the roots of i.i.d.\ copies of unconditioned $\GW_{\xi}$ random trees. (See Figure~\ref{fig:GW}.)
The random trees discussed in \cite{LeSa1997} and \cite{ChFlGaGi2004} correspond to the
special case where~$\xi$ is such that $\p{\xi=2}=\p{\xi=0}=1/2$. The results 
of \citet{Woods1997} on Cayley trees (uniformly random labelled trees) also 
fit in the framework since the shapes of a Cayley tree of size $n$ and of a Galton--Watson
tree with Poisson$(1)$ offspring conditioned on the total progeny to be $n$ have the same distribution.
(Note however, that the arguments in \cite{ChFlGaGi2004} could be probably be extended to critical Galton--Watson 
trees with an offspring distribution $\xi$ having exponential moments; here such an assumption on moment is 
not necessary.)

\begin{figure}
\begin{center}
\includegraphics[scale=.5]{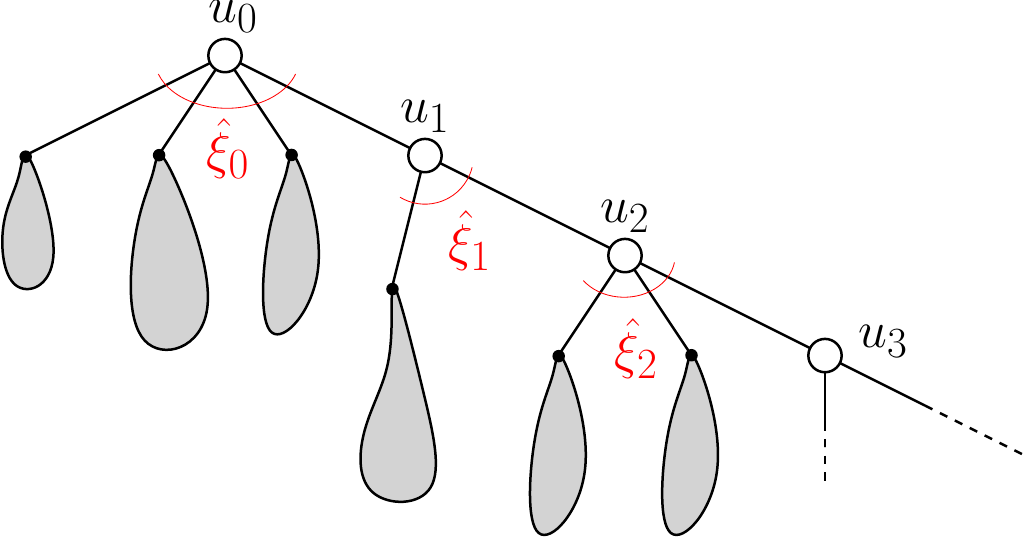}
\end{center}
\caption{The critical Galton--Watson tree of reproduction random variable $\xi$ conditioned to be infinite: the $\hat \xi_i$ are independent copies of the size biased version of $\xi$, and the grey subtrees are independent Galton--Watson trees of reproduction random variable $\xi$.}
\label{fig:GW}
\end{figure}

\subsubsection{Non-plane binary trees}
%Both `and' and `or' relations are commutative, and its seems more natural 
%to consider trees whose distribution takes these symmetries into account. 
%The following model of Boolean expressions is the topic of \cite{GeGiKrMa2013a}.
A rooted non-plane binary unlabelled tree is either a single external node, or it 
consists of an unordered pair of such trees. These trees can be seen as 
the equivalence classes of the usual (plane) binary trees where two trees are 
deemed equivalent if it is possible to transform one tree into the other 
by swapping the left and right children of a finite collection of nodes. 
These trees originate in the work of \citet{Polya1937}, and have been enumerated 
by \citet{Otter1948}; in the following we refer to them simply as \emph{unordered trees}.
They are different from the conditioned Galton--Watson trees
of the previous example in an essential way, and in particular they lack the 
nice probabilistic representation as a branching process \cite{MaMi2011a,HaMi2012a,
BrFl2012a}. 
Let $y_n$ denote the number of rooted binary unordered trees with $n$ labelled leaves. 
Then, \citet{Otter1948} 
proved that 
\begin{align}\label{eq:otter}
y_n = \kappa \rho^{-n} n^{-3/2} (1+O(1/n)),
\end{align}
for some constants $\kappa>0$ and $\rho\in (1/4,1/2)$. 
Note in particular that there is a constant $C>0$ such that $y_n\le C \rho^{-n} n^{-3/2}$ for all $n\geq 0$.
We would like to prove that if $T_n$ is a sequence of uniformly random binary unordered 
trees on $n$ leaves, then $T_n$ converges in distribution in the local sense to an 
infinite tree with a single infinite path. To prove this, we consider 
$\lambda_1^n\ge \lambda_2^n$ the sizes of the two subtrees of the root. 
For all $i<n/2$, $\lambda_2^n = i$ implies that $\lambda_1^n>\lambda_2^n$ and therefore,
there cannot be any symmetry that involves the root, implying that
\begin{align}\label{eq:otter_dist}
\p{\lambda_2^n=i}
& = \frac{y_{n-i} \cdot y_i}{y_n} = \rho^{i} y_i + O(i/n).
\end{align}
In particular, 
we have, for all $j<n/2$,
\[
\p{\lambda_2^n\ge j} 
\le \sum_{i\ge j}^{\nicefrac{n}2} \frac{y_{n-i} \cdot y_i} {y_n}
\le \frac{C^2}{y_n} \sum_{i\ge j}^{\nicefrac{n}2} \rho^{-n} (n-i)^{-3/2} i^{-3/2}
\le A j^{-1/2},
\]
for a constant $A$ and all $n$ large enough. 
This implies that $(\lambda_2^n)_{n\ge 1}$ is tight, so that by \eqref{eq:otter_dist}, it converges 
in distribution to a (real) random variable, say $X$. 

Let $(X_i)_{i\ge 0}$ be a sequence of i.i.d.\ copies of $X$, and conditional on that,
let $(Y_i)_{i\ge 0}$ denote a sequence of independent random rooted binary unordered trees
of respective sizes $X_i$. Finally, let $T_\infty$ be the binary tree consisting of a single 
infinite path $(u_i)_{i\ge 0}$ to which one appends the trees $Y_i$ by adding an 
edge between $u_i$ and the root of $Y_i$. Then, $T_\infty$ is the local weak limit 
of $T_n$. 

\subsubsection{The associative tree}\label{sub:assoc}
Suppose that for all $n\geq 1$, $G_{n,k}$ is uniformly distributed among all trees with $n$ nodes 
(instead of $n$ leaves as in the majority of examples) labelled with `and' and `or' on the internal nodes, 
and the literals $\{x_1,\bar x_1,\dots, x_k, \bar x_k\}$ on the leaves.
Let us denote by $T_{n,k}$ the random unlabelled tree obtained by forgetting the labels of $G_{n,k}$. 
Since for every $t$ with $\ell$ leaves and $n-\ell$ internal nodes, there are $(2k)^\ell$ different labellings of the leaves and $2^{n-\ell}$ labelling of the internal nodes, the probability that $T_{n,k}$ is equal to a given tree $t$ with $\ell$ leaves and $n-\ell$ internal nodes
is proportional to $k^{\ell}$.

Recall that, given a sequence of weights $(\omega_i)_{i\geq 0}$, the $n$-node simply 
generated tree is defined as follows (see, e.g., \cite{MM78}):
\begin{itemize}
\item For an $n$-node rooted tree $t$, let its weight 
$w(t)=\prod_{\nu\in t} \omega_{\mathtt{out}(\nu)}$ 
where the product is over the nodes $\nu$ of $t$ and $\mathtt{out}(u)$ denotes 
the number of children of a node $u$;
\item an $n$-node simply generated tree associated with the weight sequence 
$(\omega_i)_{i\geq 0}$ is then an $n$-node rooted tree sampled with probability 
proportional to its weight $w(t)$.
\end{itemize}

Thus, $T_{n,k}$ is the simply generated tree with weights $w_0 = k$, $w_1=0$ and
$w_i = 1$ for all $i\geq 2$.
%, meaning that $T_{n,k}$ can be obtained as follows:
%Take all trees having $n$ nodes, and weight a tree $t$ by $\prod_{\nu\in t} \omega_{\mathtt{out}(\nu)}$, 
%where the index $\nu$ of the product stands for a node of $t$ and $\mathtt{out}(\nu)$ is its out-degree.
%Then, $T_{n,k}$ is the random tree picked up at random in the set of all trees having $n$ nodes, weighted in this way.
Note that the simply generated tree with weight sequence
\begin{equation}\label{eq:GW_assoc}
\pi_i = \frac{1}{k\left(1+\frac{1}{1+\sqrt{k}}\right)} \left(\frac{\sqrt{k}}{1+\sqrt{k}}\right)^i w_i
\end{equation}
has the same law as $T_{n,k}$, and the sequence $(\pi_i)_{i\geq 1}$ is a probability sequence. 
Then $T_{n,k}$ has the same law as a (critical) Galton--Watson tree with offspring distribution $(\pi_i)_{i\geq 0}$ conditioned
on being of size $n$. 
In view of Section~\ref{sub:GW}, we know that such a tree locally converges to an infinite tree with one 
infinite end. Remark that this local limit, even unlabelled, depends on $k$.

\subsubsection{Fragmentation trees} \label{sub:alpha-gamma}
Consider a family of probability distributions $\bq:=(q_n)_{n\ge 1}$ such that $q_n$ 
is a distribution on the set of partitions of the integer $n$. 
A partition of~$n$ is a non-increasing integer sequence $(\lambda_1^n, \lambda_2^n, \ldots, \lambda_n^n)$ of sum $n$; as an example, the partitions of~$3$ are $(3), (2, 1)$ and $(1,1,1)$.
For $n=1$, we assume that $q_1$ charges the partition $(1)$, 
but also the empty sequence $\varnothing$. We require that for all $n\ge 1$, 
$q_n((n))<1$, so that $q_n$ does not only charge the partition $(n)$. 
Then the family $\bq$ induces a family of random fragmentation trees 
which are defined as the genealogical trees of the fragmentation of a collection of $n$ indistinguishable 
items, or balls. The tree $T^n$ on $n$ leaves is rooted, and this root represents
the collection of $n$ first items. With probability 
$q_n(\lambda_1,\lambda_2,\dots, \lambda_p)$, the collection $(n)$ is split into 
$p$ subcollections of sizes $\lambda_1\ge \dots\ge \lambda_p$ with $\lambda_1+\dots+\lambda_p=n$. 
The root of $T^n$ then
has $p$ children which are independent copies of $T^{\lambda_1}$, \dots, $T^{\lambda_p}$.
Note that when $q_n((n))>0$ it is possible that the collection remains unchanged, 
and that the root of $T^n$ has only one child. (There is a similar model of 
random trees with $n$ nodes, see \citet{HaMi2012a} for details.) The model emcompasses 
the ones in \cite{Aldous1996,ChFoWi2009,Devroye1998,BrDeMcSa2008}.

As we have already seen, the relevant information for and/or trees is located 
around the root, and we shall investigate conditions on $\bq$ under which a sequence of 
random fragmentation trees converges locally (in distribution). Write 
$\ell^1_\downarrow=\{(x_1,x_2,\dots): x_1\ge x_2\ge \dots \ge 0 : \sum_{i\ge 1} x_i<\infty\}$
equipped with the usual $\ell^1$ norm. 

\begin{prop}\label{prop:frag_trees}
Let $\bq=(q_n)_{n\ge 1}$ be a family of probability distributions such that $q_n$ 
is a distribution on the set of partitions of the integer $n$
with $q_n((n))<1$. Let $(\lambda_i^n)_{i\ge 1}$ be a random variable 
under $q_n$. If $(\lambda_i^n)_{i\ge 2}$ converges in distribution in $\ell^1_\downarrow$
as $n\to\infty$, then $T^n$ converges locally in distribution to a limit random 
tree $T^\infty$ with a single end.
\end{prop}

\begin{proof}
We first describe the limit tree. Let $\eta$ denote the limit distribution of 
$(\lambda_i^n)_{i\ge 2}$ under $q_n$, as $n\to\infty$. Let $(\bs^i)_{i\ge 0}$
denote of sequence of i.i.d.\ random variables distributed like $\eta$, where
$\bs^i=(s^i_j)_{j\ge 1}$, for $i\ge 0$. Note that, since the convergence holds in $\ell^1_\downarrow$, 
and since $\bs^i$ is a sequence of integers, then,
with probability one, there exists $\xi(i)<\infty$ such that $s^i_j=0$ for $j\ge \xi(i)$.
Consider the tree $T$ constructed as follows: there is a unique half-infinite 
path $u_i$, $i\ge 0$, and $T$ is rooted at $u_0$. Aside from $u_{i+1}$, the node $u_i$ 
has $\xi(i)$ extra children $v^i_j$, $j=1,\dots, \xi(i)$. Then, for $i\ge 0$ and 
$1\le j\le \xi(i)$, the node $v^i_j$ is the root of a tree $T^i_j$ which is independent 
of everything else, and distributed like a $\bq$-fragmentation tree on $s^i_j$ leaves. 
It should be clear that the tree $T$ we have just described is indeed the local 
limit of a sequence of trees $\bq$-fragmentation trees $T^n$ on $n$ leaves.
\end{proof}

We note that Proposition~\ref{prop:frag_trees} applies in particular in the case of the 
alpha-gamma model of \citet{ChFoWi2009} provided that $\gamma>0$. 
The model also encompasses the trees defined by \citet{Ford2005} and the discrete stable trees 
of \citet{Marchal2008}.

The alpha-gamma model introduced in \cite{ChFoWi2009} is a random process on the space of  
leaf-labelled trees.
There are two parameters: $\alpha\in [0,1]$ and $\gamma\in [0,\alpha]$. For $n=1$ and 
$n=2$, there is a unique $n$-leaf-labelled tree (see Figure~\ref{fig:alpha_tree}). 
Given that $T_n$ has been constructed
we assign weight $1-\alpha$
to each of the $n$ edges adjacent to the leaves, a weight $\gamma$ to each of the other 
edges and weight $(k-1)\alpha-\gamma$ to each vertex of degree $k+1\ge 3$. So there is 
a total weight of $n-\alpha$. Then, an element ---a vertex $v_n$ or an edge 
$\{a_n,c_n\}$---
is picked randomly according to the weight distribution and the tree $T_{n+1}$ is 
constructed as follows:
\begin{itemize}
	\item if we picked a vertex $v_n$, then we add the leaf $n+1$ and the edge $\{v_n,n+1\}$;
	\item if we picked an edge $\{a_n,c_n\}$ we split it into two and attach the new leaf $n+1$ 
	to the midpoint. More formally, we replace $\{a_n,c_n\}$ by three edges $\{a_n,b_n\}$, 
	$\{b_n,c_n\}$ and $\{b_n,n+1\}$.
\end{itemize}

\begin{lem}If $(\lambda_i^n)_{i\ge 1}$ is distributed 
according to $q_n^{\alpha,\gamma}$ and if $0<\gamma \leq \alpha \leq 1$,
then $(\lambda_i^n)_{i\ge 2}$ converges in distribution 
with respect to $\ell^1_\downarrow$ as $n\to\infty$.
\end{lem}
\begin{proof}
To prove convergence in distribution in $\ell^1_\downarrow$, it suffices to prove that 
$(\sum_{i\ge 2}\lambda_i^n)_{n\ge 1}$ is tight, and that $(\lambda_2^n,\lambda_3^n,\dots)_{n\ge 1}$ converges
in distribution for the product topology. 
The split distributions $(q_n^{\alpha,\gamma})_{n\ge 1}$ induced by the 
alpha-gamma model are given in Proposition 10 of \cite{ChFoWi2009}:
for all integers $k, n_1, \ldots, n_k$ such that $n_1 + \ldots, +n_k = n$,
\begin{align}
& q_n^{\alpha,\gamma}(n_1,\dots, n_k) \notag\\
& = \frac{\Gamma(1-\alpha)}{\Gamma(n-\alpha)}
\left(\gamma + \frac{1-\alpha-\gamma}{n(n-1)} \sum_{i\ne j} n_i n_j\right)
\frac {\binom{n}{n_1,\dots, n_k} } {m_1! \dots m_n!} 
\alpha^{k-2} \frac{\Gamma(k-1-\gamma/\alpha)}{\Gamma(1-\gamma/\alpha)}
\prod_{j=1}^k \frac{\Gamma(n_j-\alpha)}{\Gamma(1-\alpha)},\label{eq:alpha-gamma}
\end{align}
where $m_i=\#\{j: n_j=i\}$, for $1\le i\le n$.
For every fixed $k\ge 1$, and $(n_2,\dots, n_k)$ note that, writing 
$\tilde n=n_2+n_3+\dots+n_k$,
we can rewrite $q_n^{\alpha,\gamma}(n-n_2-\dots-n_k, n_2,\dots, n_k)$ as
$$
q_n^{\alpha,\gamma}(n-\tilde n, n_2,\dots, n_k)
= c_n 
\left(\gamma + \frac{1-\alpha-\gamma}{n(n-1)} \sum_{i\ne j} n_i n_j\right)
\frac {\alpha^{k-2}} {m_1! \dots m_n!}
\frac{\Gamma(k-1-\gamma/\alpha)}{\Gamma(1-\gamma/\alpha)}
\prod_{j=2}^k \frac{\Gamma(n_j-\alpha)}{\Gamma(1-\alpha)}
$$
where, as $n\to\infty$,
$$
c_n = \binom{n}{n-\tilde n, n_2, \dots, n_k} 
\frac{\Gamma(n-\tilde n-\alpha)}{\Gamma(n-\alpha)}
\sim \frac 1 {n_2! \cdots n_k!}.
$$
It follows that 
\begin{equation}\label{eq:qnag}
q_n^{\alpha,\gamma}(n-\tilde n, n_2,\dots, n_k)
\sim 
\frac{\gamma}{n_2! \cdots n_k!} 
\frac {\alpha^{k-2}} {m_1! \dots m_n!}
\frac{\Gamma(k-1-\gamma/\alpha)}{\Gamma(1-\gamma/\alpha)}
\prod_{j=2}^k \frac{\Gamma(n_j-\alpha)}{\Gamma(1-\alpha)}.
\end{equation}
Now to complete the proof, it suffices to check that the right-hand side above 
indeed defines a probability distribution on the set of non-increasing sequences 
of integers. 
To this aim, first observe that for all $n$ large enough,
one has $m_1!m_2!\cdots m_{n}!= m_1! m_2!\cdots m_{\tilde n}!$. Then, we can 
further rewrite \eqref{eq:qnag} as
\begin{equation}\label{eq:frag_conv-dist}
q_n^{\alpha,\gamma}(n-\tilde n, n_2,\dots, n_k)
\sim 
\mu(\tilde n) \times p_{\alpha,-\gamma}(m_1,m_2,\dots, m_{\tilde n}),
\end{equation}
where
$$
\mu(\tilde n)=\gamma \frac{\Gamma(\tilde n-\gamma)}{\Gamma(1-\gamma)\tilde n!}
$$
and
$$
p_{\alpha,\theta}(m_1,m_2,\dots, m_{\tilde n})
= \frac{\tilde n!}{\prod_{i=1}^n (i!)^{m_i}m_i!} 
\alpha^{k-1} \frac{\Gamma(k+\theta/\alpha)/\Gamma(1+\theta/\alpha)}
{\Gamma(\tilde n+\theta)/\Gamma(1+\theta)}
\prod_{j=1}^{\tilde n} \left( \frac{\Gamma(j-\alpha)}{\Gamma(1-\alpha)}\right)^{m_j}.
$$
Now, $p_{\alpha,\theta}(m_1,\dots, m_{\tilde n})$ is the probability distribution 
associated to Pitman's generalization of the Ewens sampling formula, see Proposition~9
of \cite{Pitman1995a}. Finally, for any $\gamma>0$,
$$
\sum_{\tilde n\ge  1} \mu(\tilde n) 
= \sum_{\tilde n\ge 1}
\gamma \frac{\Gamma(\tilde n-\gamma)}{\Gamma(1-\gamma)\tilde n!}
= \sum_{\tilde n \ge  1}
\frac{\Gamma(\tilde n-\gamma)}{-\Gamma(-\gamma)\tilde n!}
= 1
$$
and $(\mu(i))_{i\ge 0}$ is also a probability distribution related to
the negative binomial distribution. This shows in particular that 
$(\sum_{i\ge 2} \lambda_i^n)_{n\ge 1}$ converges in distribution to $\mu(\tilde n)$, 
which, together with \eqref{eq:frag_conv-dist}, completes the proof.
\end{proof}

\section{Galton--Watson trees: improved Lefmann and Savick{\'y} bounds}\label{sec:improved_bounds}

In this section we obtain the Lefmann and Savick{\'y}~\cite{LeSa1997} bounds by a branching argument, 
improve them via a very simple symmetry argument, and extend them to all Galton--Watson trees with 
progenies having exponential tails. 
This is also the occasion to introduce the \emph{trimming procedure} that will be crucial in the 
remainder of the paper (Section~\ref{sec:bounds}).

\subsection{A symmetry argument}
Let us recall
Equation~\eqref{eq:LeSa-bounds}, by~\citet{ChFlGaGi2004} which states that 
if $T_n$ is uniformly distributed among binary trees having $n$ leaves,
then there exist two constants $c_1,c_2>0$, 
for all $k\geq 1$, for any $k$-variable Boolean function~$f$
\begin{equation}\label{eq:general_bounds}
\exp(-c_1 L(f) \log k) \le P_k[T_n](f) \le \exp(- c_2 L(f) k^{-2}).
\end{equation}
The bound in \eqref{eq:general_bounds} suffers from two main problems: first
it does not constrain in any way the probability of
functions of complexity of order $o(k^2)$;
and second, it has only been proved for the case of binary trees. 
In the following, we tighten the upper bound and generalize it to any critical Galton--Watson tree
conditioned on being infinite, under the condition that the offspring distribution has 
exponential moments. 
A simple observation also allows us to strengthen the 
upper bound above.

First we need to define the notion of essential variables: 
\begin{defn}
Fix an integer $k$.
Let $f$ be a $k$-variable Boolean function: $f\;:\; (x_1, \ldots, x_k)\mapsto f(x_1, \ldots, x_k)$. 
For all $1\leq i\leq k$, we say that the variable $x_i$ is {\bf essential} for $f$ if and only if
$f_{|x_i=0} \neq f_{|x_i=1}$ 
(meaning that the restriction of $f$ to the subspace where $x_i=1$ is not the same Boolean function as
the restriction of $f$ to the subspace where $x_i=0$).
\end{defn}

\begin{thm}\label{thm:improved_gen_bound}
Suppose that $\Ec{\xi}=1$, $\pc{\xi=1}=0$, and that there exists $a>0$ such that $\Ec{e^{a \xi}}<\infty$. 
Let $T_\infty$ be a Galton--Watson tree with offspring distribution $\xi$, conditioned on being infinite. 
Then, there exists constants $c_1,c_2>0$ such that, 
for every $k\ge 2$,
for any $k$-variable Boolean function~$f$, we have 
\begin{equation}\label{eq:LeSa-bounds-improved}
\exp(-c_1 L(f) \log k) 
\leq P_k[T_\infty](f) 
\leq \exp\left(- c_2 L(f) k^{-2} - \log \binom k {\Ess(f)}\right),
\end{equation}
where $\Ess(f)$ is the number of essential variables of $f$.
%Furthermore, the upper bound cannot be significantly improved. 
\end{thm}

This section is devoted to achieving two main goals: we first prove by a symmetry argument 
the improved upper bound of Theorem~\ref{thm:improved_gen_bound} relying on the looser upper 
bound proved by~\citet{ChFlGaGi2004}; we then propose in Section~\ref{sub:alt_proof} a simpler 
and more general proof for~\citet{ChFlGaGi2004}'s result that applies to a wider class 
of Galton--Watson trees.

Let us first improve \citet{ChFlGaGi2004}'s upper bound.
Their proof relies on the apparently blunt upper bound 
\[
\bbP(\of[\tau] = f)\leq \bbP(\|\tau\|\geq L(f)),
\]
for all $k$-variable Boolean trees $\tau$ and all $k$-variable Boolean functions.
This inequality is loose since $\p{\|\tau\|\ge L(f)}$ actually 
bounds the probability of the collection of functions which may be obtained 
from $f$ by permutations of the variables. 
More precisely, consider any Boolean function $f$ with $\gamma:=\Ess(f)$ essential variables.
Assume without loss of generality that the essential variables of $f$ are $x_1, \ldots, x_{\gamma}$.
Consider a Boolean function $f$ and a permutation $\pi$ of $(x_1, \ldots, x_k)$ 
which does not map $\{x_1, \ldots, x_{\gamma}\}$ into itself. 
Then, by symmetry, the Boolean function $f\circ \pi$ has the same probability as $f$. 
Furthermore, for any such permutation $\pi$ the functions $f\circ \pi$ 
and $f$ are distinct. Thus, we actually have, for all $k\geq 1$
\[
\binom{k}{\Ess(f)} \bbP(\of[\tau] = f)
\leq \bbP(\|\tau\|\geq L(f)) 
\leq \exp(- c_2 L(f) k^{-2}).
\]
Applying this inequality to $\tau = T_{\infty}$, and
%which implies that, 
writing $\mathbb P(f):=\bbP_k[T_\infty](f)$, we get
\[\bbP(f) \leq 
\exp\left(- c_2 L(f) k^{-2} - \log \binom{k}{\Ess(f)}\right).
\]
Moreover, it is interesting to note that this new upper-bound happens to be optimal in some 
cases, at least at the level of exponents, since it is achieved for read-once functions 
(see~\cite[page~25]{BooleanFunctionsBook} for a definition of such functions):
Consider a %sequence of 
$k$-variable read-once function~$f$,
meaning in particular that $L(f)=\Ess(f)$, 
and suppose further that $L(f)=\Ess(f)= o(k)$. 
Then, in view of Equation~\eqref{eq:LeSa-bounds-improved}, there exists a constant $c_3$ such that
$$
\exp(-c_1 L(f) \log k) 
\leq \bbP(f) \leq 
\exp(-c_3 L(f) \log k),
$$
so that, neither the upper nor the lower bound can be significantly improved without 
considering exponents that would depend on other parameters than the mere 
complexity~$L(f)$.

\subsection{Proof of Theorem~\ref{thm:improved_gen_bound}: the trimming procedure}\label{sub:alt_proof}
The lower bound is exactly the one from Theorem~1.1 of \cite{LeSa1997} and we shall not reproduce the argument.
Our proof of the upper bound relies on a refined analysis of a certain 
trimming procedure, which removes the portions of the tree that do not influence the Boolean 
function it encodes. The cutting procedure is similar to the one used by \citet{ChFlGaGi2004}, 
but modified in order to simplify the analysis but also to make it more powerful. 
(The interested reader can easily verify 
that our procedure removes more nodes than the one in~\cite{ChFlGaGi2004}.) 

% In this section, we define a slightly different trimming procedure. So far the procedure
% that has been used by \citet{ChFlGaGi2004} was sufficient for our needs. It is no 
% longer the case, and we now introduce the following modified version, which turns out to 
% have much nicer properties. 
Consider an and/or tree $\tau$. 
Let $\diamond_u$ denote the label of a node $u$. 
We also let $\Lambda(u)$ denote the collection of children of $u$ that are leaves.
Given a leaf $w$, we denote its label by $\mathtt{Lab}(w)$.
We associate a set of {\bf constraints} to every node, 
such that the sequences of sets seen when following paths 
away from the root are increasing (for the inclusion). 
Inductively define the constraints sets for all node using the following rules:
% Fix $\tau$ a possibly infinite and/or tree.
% We first recall our notations:
%Let~$\diamond_u$ denote the label of a node $u$ in $\hat t_{\înfty}$ (the randomly labelled version of $t_\infty$). 
% For any internal node of $\tau$, we let $\Lambda(u)$ be the collection of children of $u$ in $\tau$ that are leaves.
%For a subtree $t$ of $\tau$, we denote by $\Lambda(t)$ the set of leaves of $t$. 
% We denote by $\mathtt{Lab}(w)$ the label of the leaf $w$ of $\tau$.
% As before, we now associate a set of constraints to every node, 
% such that the sequences of sets seen when following paths 
% away from the root are increasing (for the inclusion). 
% Inductively define the constraints sets for all nodes using the following rules:
for the root $r$, we set $C_r:=\emptyset$;
for a node $v$ that is a child of $u$,
\begin{itemize}
	\item if $u$ is labelled by `and' then 
	$C_v = C_u \cup \bigcup_{w\in \Lambda(u)} \{\Lab(w)=\True\}$;
	\item if $u$ is labelled by `or' then 
	$C_v=C_u \cup \bigcup_{w\in \Lambda(u)} \{\Lab(w)=\False\}$.
\end{itemize}
We say that a node is {\bf consistent} if there exists an assignment 
of the variables satisfying all the constraints of its constraint set;
note that all descendants of a non-consistent node are also non-consistent.
We denote by $\trim(\tau)$ the labelled tree
obtained from $\tau$, by keeping only the nodes that are consistent 
(an example is given in Figure~\ref{fig:ex_trim}).
For a subtree $\tau'$ of $\tau$, we let $\trim(\tau')$ denote the portion of $\trim(\tau)$
that is in $\tau'$ (with its labels).

\begin{figure}
\begin{center}
\includegraphics[width=.7\textwidth]{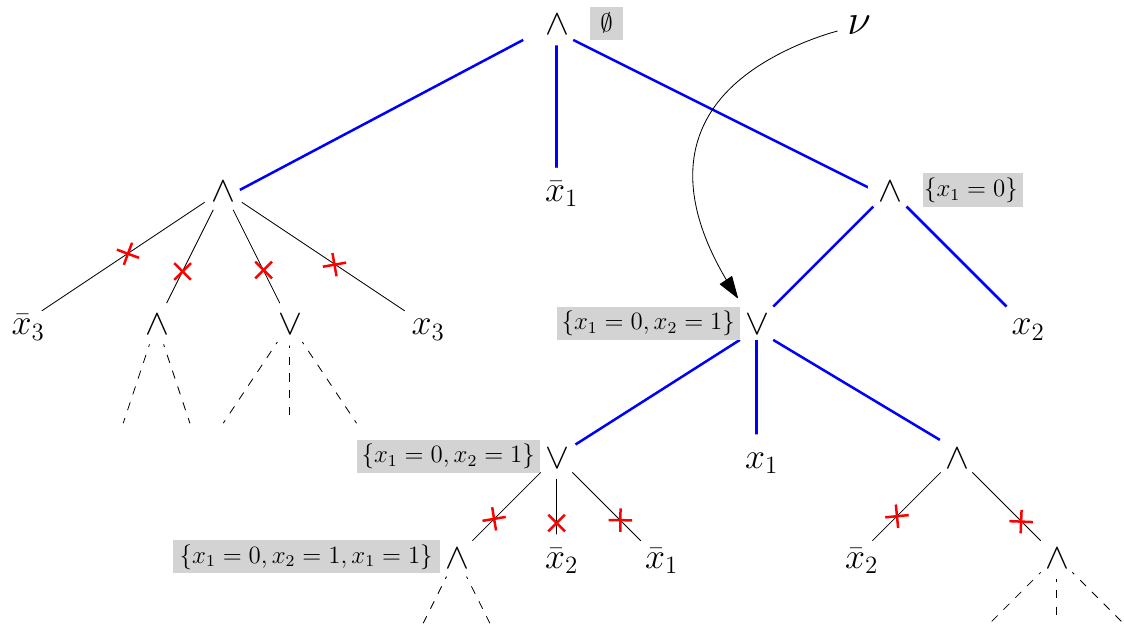}
\caption{The above tree is an example of an and/or tree $\tau$ to which we apply 
the trimming procedure $\trim$. The dashed parts are possibly infinite and/or trees. 
The tree $\trim(\tau)$ is highlighted in bold and in blue. 
The constraint sets of some nodes have been added in the shaded boxes; 
note that due to space constraints, $\True$ is represented by $1$ and $\False$ by $0$.
For example, consider the node $\nu$ above. The parent of~$\nu$ has constraint set $\{x_1 = 0\}$,
and a sibling of $\nu$ is a leaf labelled by the literal $x_2$, 
then, since the parent of $\nu$ is labelled by $\land$, its constraint set is $\{x_1 = 0\}\cup \{x_2 = 1\}$.}
\label{fig:ex_trim}
\end{center}
\end{figure}

In the procedure of \citet{ChFlGaGi2004}
the label of the leaf $v$ does not affect the constraint set of $v$ itself.
This is now possible in this improved trimming procedure.
% It is straightforward that for all and/or tree $\tau$, $\trim^+(\tau)\subseteq \trim(\tau)$ 
% (see Figures~\ref{fig:ex_trim_normal} and~\ref{fig:ex_trim} for an example). 
Also, as a consequence of the definition, if any node $v$ is inconsistent, so are all 
its siblings. It follows that some internal nodes of $\tau$ end up having no 
progeny in $\trim(\tau)$ (see Figure~\ref{fig:ex_trim}), and thus become leaves of $\trim(\tau)$. 
For these nodes, we adopt the convention that a leaf of $\trim(\tau)$ that is 
an internal node in~$\tau$ labelled by~$\land$ (resp.~$\lor$) has Boolean value 
$\False$ (resp. $\True$). We first verify that this trimming procedure does not 
modify the Boolean function that the tree represents:

\begin{lem}\label{lem:trim+}
For every and/or tree $\tau$, 
$\trim(\tau)$ calculates the same Boolean function as $\tau$.
\end{lem}

\begin{proof}
Let $\nu$ be an inconsistent node of $\tau$. 
Let us prove that the tree obtained by cutting $\nu$ and all its progeny from $\tau$ 
calculates the same Boolean function as $\tau$.
The fact that $\nu$ is inconsistent means that 
there exist two internal nodes $\nu_1$ and $\nu_2$ 
on the path between the root of $\tau$ and $\nu$ such that
assigning $x$ to $\False$ makes the tree rooted at $\nu_1$ calculate a constant function 
(more precisely $\True$ if $\nu_1$ is labelled by $\lor$ or $\False$ otherwise)
and assigning $x$ to $\True$ makes the tree rooted at $\nu_2$ calculate a constant function.
Note that the restriction on $\{x=\False\}$ (resp. $\{x=\True\}$) of the Boolean 
function calculated by $\tau$ and of the Boolean function calculated by the tree obtained 
from $\tau$ by cutting all progeny of $\nu_1$ (resp.~$\nu_2$) are equal.
Let us assume for example that $\nu_1$ is an ancestor of $\nu_2$ 
(note that our reasoning also holds when $\nu_1=\nu_2$). 
Then, the tree obtained by cutting all progeny of $\nu_2$ calculates the same Boolean 
function as $\tau$, implying the result.
\end{proof}

Finally, we also define the {\bf size} of $\trim(\tau)$ as the number of its leaves that 
are labelled by literals in $\{x_1,\bar x_1, \dots,x_k, \bar x_k\}$ (or the number of 
leaves that were already leaves in $\tau$). With this definition, we see that 
the functions $\True$ and $\False$ are computed by trees of size zero (a single 
internal node labelled by $\lor$ or $\land$), which agrees with our previous 
convention that they should have complexity $L(\True)=L(\False)=0$.

\begin{prop}\label{pro:bound-trimmed-GW}
Let $\xi$ be an integer-valued random variable such that $\Ec\xi=1$ and $\pc{\xi=1} = 0$. 
Suppose further that there
exists an $a>0$ such that $\Ec{e^{a \xi}}<\infty$. Let $T_\infty$ be a Galton--Watson tree 
with offspring distribution~$\xi$, conditioned on being infinite.
Recall that $\hat T_{\infty}$ stands for the randomly $k$-variable labelled version of $T_{\infty}$.
%, and let $\trim(\hat T_{\infty})$ denote the $k$-trimmed subtree of $\hat T_\infty$. 
Then, there exists a constant $c>0$ such that, for any $\varpi>0$, for any integer $k\geq 1$, we have
\[
\p{\|\trim(\hat T_{\infty})\| \ge  \varpi} \le \exp(-c  \varpi/k^2).
\]
\end{prop}
The rest of the section is devoted to proving this proposition.

A node $u$ may have multiple children, 
and the set of constraints of its children may be inconsistent even if $C(u)=\emptyset$. 
However, to simplify the analysis we will only search for inconsistencies at the children of 
nodes $u$ for which we already have $C(u)\ne \emptyset$. 
In order to bound the size of the trimmed portion of a tree, 
we decompose the tree into a (maximal) subtree which contains only 
nodes with empty sets of constraints (together with the leaves that may be attached to it), 
to which are grafted subtrees whose internal nodes have non-empty constraint sets. 

Consider $T_{\infty}$, a Galton--Watson tree with critical offspring distribution $\xi$ 
conditioned to be infinite. This is the random tree we have 
introduced in Section~\ref{sec:examples}.
Recall that a critical Galton--Watson tree conditioned on being infinite can be 
obtained by size-biasing the progeny distribution the nodes on a single infinite 
path from the root:
Write $(\hat\xi_i)_{i\ge 0}$ for a sequence of 
i.i.d.\ copies of $\hat \xi$, the size-biased version of $\xi$ 
(recall that $\pc{\hat\xi=i}= i \p{\xi=i}$). 
Then $T_{\infty}$ consists of an infinite backbone $(u_i)_{i\ge 0}$ 
such that the node $u_i$ has $\hat\xi_i$ children,
one of which is $u_{i+1}$. All the other offsprings of the nodes $u_i$, $i\ge 0$, are the 
roots of independent (unconditioned) $\GW_\xi$ random trees (see Figure~\ref{fig:GW}).

Let us first consider $T$, an unconditioned Galton--Watson tree of progeny distribution $\xi$, 
and prove that $\|\trim(\hat T)\|$ has exponential tails:
\begin{prop}\label{prop:uncond}
Let $T$ be an unconditioned Galton--Watson tree of progeny distribution $\xi$, 
such that there exists a constant $a>0$ verifying $\Ec{e^{a\xi}}<\infty$, and such that $\mathbb P(\xi=1)=0$.
Then, $\|\trim(\hat T)\|$ has exponential tails.
\end{prop}

To prove Proposition~\ref{prop:uncond}, let us colour the nodes of $\trim(\hat T)$ as follows: 
the nodes having an empty constraint set in blue, 
and the nodes having a non empty constraint set in red.
It is enough to prove that the random sizes of all these different clusters, which we denote respectively by
$\|\mathtt{blue}\|$ and $\|\mathtt{red}\|$ have exponential tails.
Note that the different red clusters are i.i.d. and therefore, their sizes all have the same law: 
$\|\mathtt{red}\|$ is the size of one of them and not the size of their union.
Recall that the random labelling of~$T$ is defined for a certain integer~$k$,
and that $\|\trim(\hat T)\|$, $\|\mathtt{red}\|$ and $\|\mathtt{blue}\|$ thus depend on~$k$.

The following two lemmas concern each of the two coloured random trees.

%We now prove that the $k$-trimmed subtree of $\hat \tau$, denoted by $\trim(\hat \tau)$ is contained 
%in a branching process which is sufficiently subcritical to imply the tail bounds 
%for its size $\|\trim(\hat \tau)\|$. 
%
%Bounding the sizes of $\tau_0$ and $\tau_1$ is done in a similar way. 
%We detail the case of $\tau_1$. 
%The case of $\tau_0$ is then only sketched. 

\begin{lem}\label{lem:red}
Let $\xi$ be an integer-valued random variable such that $\Ec \xi=1$ and $\pc{\xi=1}=0$. Suppose that 
there exists a constant $a>0$ such that $\Ec{e^{a\xi}}<\infty$. 
Then, there exists a constant $c>0$ such that, for all integers~$k$ and~$m$,
\[
\p{\|\mathtt{red}\|\ge m} \le \exp(-c m/k^2)
\]
\end{lem}
\begin{proof}
Note that a red cluster is a Galton--Watson tree of progeny $\xi$, whose root has a non-empty constraint set.
Let us introduce an alternative way of sampling the $k$-trimmed 
version. We do this by introducing two types of nodes, say black and white. 
The root of the tree is black, meaning that its constraint set is non-empty: 
assume without loss of generality that this constraint set contains $\{x_1 = \True\}$.
In the following the black nodes will be the internal nodes and the white nodes will be the leaves.
Note that every black node has a constraint set that contains $\{x_1 = \True\}$.
The white nodes are leaves and have no children, 
they give a chance to their siblings to become inconsistent:
a leaf labelled by $x_1$ (resp. $\bar x_1$) and whose father is labelled by $\lor$ (resp. $\land$)
makes its siblings inconsistent. 
\begin{figure}[h]
\begin{center}
\includegraphics[width=8cm]{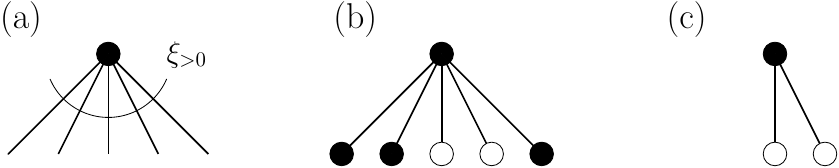}
\caption{Reproduction of black nodes: 
(a) first sample the number of children with distribution the law of $\xi_{>0}$,
(b) colour the children independently at random in white with probability $p_0$ or black otherwise,
(c) if at least one of the children is white then, with probability $1/(2k)$, delete all the black children.}
\label{fig:lem46}
\end{center}
\end{figure}
The black nodes reproduce as follows (see Figure~\ref{fig:lem46}): 
first sample $\xi_{>0}$, a copy of $\xi$ conditioned on $\xi\ge 1$. Then 
colour each node white with probability $p_0=\p{\xi=0}$, or black with probability $1-p_0$. 
Now comes the trimming part: if there is at least one white child, the black siblings 
are all removed with probability $\nicefrac1{2k}$.
One easily verifies that the tree obtained by this
branching process is stochastically larger than the tree obtained by applying 
the $k$-trimming procedure described above (in the $k$-trimming procedure,
the constraint-sets of the black nodes possibly contain more than one constraint, 
and each white node gives a chance to trim all its siblings including itself, and 
not only its black siblings). 
The matrix $M$ of mean offspring of this branching process is 
$$
M=\begin{bmatrix}
0 & 0 \\
\Ec W & \Ec B
\end{bmatrix},
$$
where $W$ and $B$ denote respectively the number of white and black children of a black 
node. In particular, the largest eigenvalue $M$ is $\mathbb E[B]$, and it is this value that characterises 
the asymptotic behaviour of the branching process.

We have
\begin{equation}\label{eq:def_dist-gw}
W\sim\bin(\xi_{>0}, p_0) 
\qquad\text{and}\qquad
B=\frac{\I{W>0}}{2k}\,\delta_0 + \left(1-\frac{\I{W>0}}{2k}\right)(\xi_{>0}-W),
\end{equation}
where $\bin(n,p)$ stands for the binomial distribution of parameters $n\in \mathbb N$ and $p\in[0,1]$,
and $\delta_0$ is the Dirac mass at~$1$.
Let $\bar W = \xi_{>0} - W$. Note that this new random variable has law $\bin(\xi_{>0}, 1-p_0)$, and that 
\[B=\frac{\I{\bar W < \xi_{>0}}}{2k} \delta_0 + \left(1-\frac{\I{\bar W < \xi_{>0}}}{2k}\right) \bar W.\]
Thus,
\begin{align*}
\Ec{B} 
& = \mathbb E \bar W - \frac1{2k} \mathbb E[\bar W \I{\bar W < \xi_{>0}}]\\
&= (1-p_0)\mathbb E \xi_{>0} - \frac1{2k} \left(\mathbb E \bar W - \mathbb E [\bar W \I{\bar W=\xi_{>0}}]\right)\\
&= \frac{1-p_0}{1-p_0} - \frac1{2k}\left((1-p_0)\mathbb E\xi_{>0} - \mathbb E[\xi_{>0}(1-p_0)^{\xi_{>0}}]\right)\\
&= 1- \frac1{2k} \left(1 - \mathbb E[\xi_{>0}(1-p_0)^{\xi_{>0}}]\right),
\end{align*}
because $\mathbb E \xi_{>0} = \frac1{1-p_0}$ since $\mathbb E \xi = 1$.
Note that, with $p_i=\pc{\xi=i}$, we have
\[\mathbb E[\xi_{>0}(1-p_0)^{\xi_{>0}}]
= \sum_{i\geq 1} ip_i (1-p_0)^{i-1} < 1,\]
since $p_1= \mathbb P(\xi=1)=0$ by assumption.
We thus have that there exists a positive constant~$c$ such that,
\[\mathbb E B = 1 - \frac{c}{k}.\]
It follows that the black subtree rooted at a black node is a subcritical Galton--Watson
tree. Let $\Xi$ denote the total progeny of the black subtree, and $(B_i)_{i\ge 1}$ be 
i.i.d.\ copies of $B$. Then,
\[
\p{\Xi \ge n } 
 \le \p{\sum_{i=1}^n (B_i-1) \ge 0}
 \le \exp(-n \Psi^\star(1-\Ec{B})),
\]
by Cram\'er's theorem \cite{DeZe1998}, where $\Psi^\star$ is the large deviations rate function 
of $B-\Ec B$. More precisely, write $\Psi(t):=\log \Ec{e^{t (B-\Ec B)}}$, 
and observe that $\Psi(t)<+\infty$ since $\xi$ has exponential moments, 
and $\Psi(\varpi)\sim \Psi''(0)  \varpi^2/2$ as $\varpi\to 0$. 
Then we have, setting $\alpha = 1/\Psi''(0)$, as $\varpi\to 0$, 
\begin{align*}
\Psi^\star(\varpi) = \sup_{t\in \R} \{t \varpi - \Psi(t)\} \ge \alpha \varpi^2 - \Psi(\alpha \varpi) 
\sim x \varpi^2 \alpha/2~.
\end{align*}
This guarantees that there exists a constant $c>0$ such that 
\begin{equation}\label{eq:tail_subcrit-GW}
\p{\Xi \ge n } \le e^{-c n/k^2}.
\end{equation}
In order to recover the size of the whole tree, it suffices to add the missing white
nodes. Writing $(W_i)_{i\ge 1}$ for a sequence of i.i.d\ copies of $W$ defined in \eqref{eq:def_dist-gw}, 
it follows from \eqref{eq:tail_subcrit-GW} that 
\[
\p{\|\mathtt{red}\| \ge m }
\le \p{\sum_{i\ge 1} W_i \I{i\le \Xi}\ge m} 
\le \p{\sum_{1\le i\le 2 m/\Ec{W} } W_i \ge m} + e^{-c_2  m/(k^2 \Ec B) }
\le e^{-c_1 m} + e^{-c_2 m/k^2},
\]
by a second use of Cram\'er's theorem for large deviations. 
Indeed, since $\xi$ has exponential tails, $W$ also does and the theorem applies.   
\end{proof}

\begin{lem}\label{lem:blue}
Let $\xi$ be an integer-valued random variable such that $\Ec \xi=1$, and $\pc{\xi=1}=0$. Suppose that 
there exists a constant $a>0$ such that $\Ec{e^{a\xi}}<\infty$. 
Then, there exists a constant $c>0$ such that, for all integers $k$ and $m$,
\[
\p{\|\mathtt{blue}\|\ge m} \le \exp(-c m/k^2).
\]
\end{lem}
\begin{proof}
The proof is very similar to that of Lemma~\ref{lem:red}.
We now have black, white, and in addition green nodes which are the ones having a white sibling. 
Black nodes are the ones having an empty constraint set, the white ones are the leaves 
and the green ones are nodes having a non-empty constraint set 
(being then the roots of independent copies of a red cluster).
Clearly, the black subtree is dominated by the black subtree considered in the proof of Lemma~\ref{lem:red}, 
so its size also has exponential tails. 
Then the nodes to add are only the white and green nodes, and the distribution is precisely that 
of $\xi_{>0}$ conditioned on $W>0$. This also has exponential tails since 
\[\p{\xi_{>0}\ge \varpi \,|\, W>0} \le p_0^{-1} \p{\xi_{>0} \ge 0}.\]
We omit the straightforward details.
\end{proof}

\begin{wrapfigure}{r}{6cm}
\includegraphics[width=4.5cm]{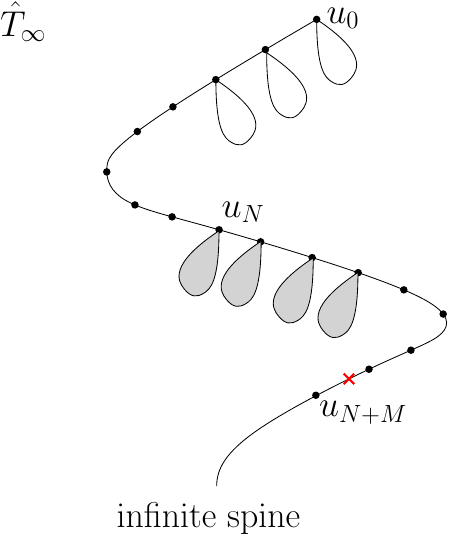}
\caption{The first $N$ nodes of the spine have empty constraint sets, the following $M$ nodes have non-empty constraint sets, 
and all nodes below that are non-consistent, and thus get cut during the trimming procedure. The white subtrees fall under Proposition~\ref{prop:uncond}, the gray ones under Lemma~\ref{lem:red}.}
\label{fig:prop44}
\end{wrapfigure}
Lemmas~\ref{lem:red} and~\ref{lem:blue} tell us that each red subtree has exponential tails, and that their number (equal to $\|\mathtt{blue}\|$) also has exponential tails, which implies Proposition~\ref{prop:uncond}.

\begin{proof}[Proof of Proposition~\ref{pro:bound-trimmed-GW}]
Let $T_{\infty}$ be a Galton--Watson tree of progeny $\xi$ conditioned on being infinite.
It can be described as an infinite spine on which independent (unconditioned) 
Galton--Watson of progeny distribution $\xi$ are hanging.
Let us associate to every node of $\hat T_{\infty}$ its constraint set as explained in the trimming procedure, 
and let us first focus on the nodes of the spine $(u_i)_{i\geq 0}$ (see Figure~\ref{fig:prop44}). 
The first nodes of this sequence have empty constraint sets. 
Therefore, the trees hanging onto them fall under Proposition~\ref{prop:uncond}.
The following nodes on the spine have non-empty constraint sets and therefore,
the unconditioned Galton--Watson trees hanging on them have the same law as the red clusters studied in Lemma~\ref{lem:red}.
In all cases, the trimmed versions of the subtrees hanging on the spine have exponential tails.

It thus only remains to prove that the total number of trees hanging on the trimmed spine has exponential tails.
Let us denote by $\hat W_i$ the random number of leaves of node $u_i$. 
The sequence $(\hat W_i)_{i\geq 0}$ is i.i.d. 
and we denote by $\hat W$ a random variable having the common law of the $\hat W_i$'s.
Recall that we have two different kinds of nodes on the spines: 
nodes with empty constraints sets (at the top of the tree), 
and nodes with non-empty constraints sets.
%whose father belongs to the spine. 
%In order to prove that $\trim(\hat T_{\infty})$ has the claimed tails, 
%it is enough to control the tails of $\tau_0$ and 
%of $\tau_1$ when they are conditioned to be infinite. 
%We develop the case of bounding $\tau_1$ when the tree is conditioned to be infinite, which relies 
%on the tail bound for $\tau_1$ when the tree is unconditioned. Then, since the nodes pending off 
%$\tau_0$ are all copies of $\tau_1$, and that the upper tails for the size do not depend on 
%whether the tree is conditioned to be infinite or not, the very same argument implies that the 
%entire trimmed tree has the required tails.
Let us denote by $N$ the number of nodes on the spine with empty constraint sets, 
and by $N+M$ the total number of nodes on the trimmed spine.
Let us prove that the total number of trees hanging on nodes of the spine with non-empty constraint sets has exponential tails.
The proof that the total number of trees hanging on nodes of the spine having \emph{empty} constraints set has exponential tails 
follows the same outline and is actually simpler: this case will be left to the reader.

So let us treat the case of the $M$ nodes of the trimmed spine having non-empty constraint sets:
let $\nu$ be such a node.
Assume without loss of generality that its constraint set is $\{x_1 = \True\}$.
Since the spine has not been cut before level $N+M$, the node $\nu$ cannot be inconsistent: 
thus the leaf-children of $\nu$ cannot be labelled by $x_1$ 
(resp. $\bar x_1$ depending on the connector labelling $\nu$).
We denote by $\hat S$ the indicator of the event ``no leaf-child of $\nu$ is labelled by $x_1$'', 
and by $\hat W$ the number of leaf-children of $\nu$. We have:
\[
\hat W=\bin(\hat\xi - 1, p_0), \qquad \text{ and }\qquad
\hat S=\ber\big((1-\nicefrac1{2k})^{\hat W}\big).
\]
Therefore, if we denote by $m:=\min\{n\geq 1\ |\ p_n>0\}$, where $p_n:=\mathbb P(\xi=n)$, then
\[\pc{\hat S=1} = \mathbb E[(1-\nicefrac1{2k})^{\hat W}]
\geq \mathbb P(\hat W=0) \geq \mathbb P(\hat \xi=m) (1-p_0)^{m-1} =: \kappa > 0,\] 
and it follows that, $i\ge 0$ and $k\ge 1$,
\[
\pc{\hat \xi-1\ge i\,|\, \hat S=1} 
\le \frac{\pc{\hat \xi-1\ge i}}{\pc{\hat S=1}}
\le \kappa^{-1} \pc{\hat \xi\ge i+1}.
\]
In particular, if $\hat \xi$ has exponential moments, so does the number of siblings 
of a node $u_i$ of the spine conditional on not being cut. 
Let $(X_i)_{i\ge 1}$ be an i.i.d.\ sequence of copies of $\hat \xi-1$ conditional on $\hat S=1$,
and $M$ for a geometric random variable with success parameter $\eta:=\pc{\hat S=0}<1$ 
(since $\pc{\hat W>0}>0$). 
Then, writing $\Delta$ for the collection of trees hanging on node of the trimmed spine 
having a non-empty constraint set, we have, for all integer $m$:
\begin{align}\label{eq:number_roots}
\p{|\Delta| \ge m} 
& \le \p{\sum_{i\ge 1} X_i \I{i \le M}\ge m} \notag \\
& \le \p{\sum_{i=1}^m X_i \ge n} + \p{M\le m}\notag \\
& \le \p{\sum_{i=1}^{\lceil \epsilon m \rceil } X_i \ge n} + (1-\eta)^{\epsilon m} \notag \\
& \le e^{-c m},
\end{align}
for $\epsilon>0$ small enough and $c$ a constant depending on $\epsilon$, which concludes the proof.
\end{proof}

\section{Improved relations between complexity and probability}
\label{sec:bounds}
%We have proven in Section~\ref{sec:convergence} the convergence in distribution of the random Boolean function 
%to an asymptotic distribution $P_k$
%when $n$ tends to infinity, and Section~\ref{sec:improved_bounds} was devoted to 
%bound $P_k(f_k)$ for any sequence $(f_k)_{k\geq 1}$ of Boolean functions on $k$ variables in the case of the Catalan tree.

Section~\ref{sec:convergence} was devoted to proving that, when the sequence of random trees $(T_n)_{n\geq 1}$
converges locally in distribution to an infinite tree with finitely many ends $T_{\infty}$,
then the distribution of the random Boolean function calculated by the random $k$-variable labelling of $T_n$ 
converges to an asymptotic distribution $\mathbb P_k = P_k[T_{\infty}]$ when $n\to+\infty$ (cf.~Theorem~\ref{thm:cv}).

We here state and prove an equivalent of the result of~\citet{Kozik2008a} (see Equation~\eqref{eq:kozik-theta}): 
fix an integer $k_0$, and a $k_0$-variable Boolean function~$f$,
we are able to understand the behaviour of $\mathbb P_k(f)$ when $k$ tends to infinity.
Remark though that we need stronger assumptions than the one needed to get convergence to the 
asymptotic distribution: ({\sc a}) We restrict ourselves to random trees whose local limit has a 
unique end, although we believe that the result holds for local limits having finitely many ends.
More importantly, ({\sc b}) we need assumptions that permit to control the sizes of the finite 
trees attached to the infinite spine.

\subsection{Controlling the repetitions}\label{sec:repetitions}
First of all, let us prove the following crucial lemma concerning the probability that the $k$-trimmed version 
(according to~$\trim$) of a randomly labelled tree contains repetitions.
The number of {\bf repetitions} in a labelled and/or tree is defined as 
the difference between the number of its leaves and the number of distinct variables that appear as leaf-labels of this tree. 
As an example, the left tree in Figure~\ref{fig:example_Btree} has~5 
repetitions since it has size~8 and is labelled by~3 distinct variables, namely~$x_1, x_2$ and~$x_3$.

\begin{lem}\label{lem:trim}
\begin{enumerate}[(a)]
\item There exists an integer $k_0$ such that, for any $k\geq k_0$, 
for any tree $t_{\infty}$, for any integer~$q$,
\[\mathbb P(\trim(\hat t_{\infty}) \text{ contains at least }q\text{ repetitions})
\leq \frac{\Ec{\|\trim(\hat t_{\infty})\|^{2q}}+2\mathtt e\ \Ec{\|\trim(\hat t_{\infty})\|^q}}{k^q}.\]
\item For any infinite tree $t_{\infty}$, and for all integers $p$ and $q$, there exists 
a constant $K_{p,q}(t_{\infty})$ such that
\[\displaystyle \mathbb P(\trim(\hat t_{\infty}) \text{ has size $p$ and contains at least }q\text{ repetitions})
\leq \frac{K_{p,q}(t_{\infty})}{k^{q+1}}.\]
\end{enumerate}
\end{lem}

\newcommand{\shape}{\mathtt{shape}}
\newcommand{\sub}{\mathtt{sub}}
\newcommand{\lab}{\mathtt{lab}}
\newcommand{\nt}{\mathtt{notrim}}
\begin{proof}
%Again, the entire argument will be conditional on the labelled tree $\hat t_\infty$. 
$(a)$
For a subtree $t$ of \cec$\hat t_\infty$\fincec, we denote by $\shape(t)$ the tree $t$ in which 
the nodes that are leaves of $\hat t_\infty$ have been unlabelled. We emphasize the fact 
that an element $\shape(t)$ may have leaves labelled by connectives ($\land$ or $\lor$), 
and that all its internal nodes are labelled by connectives.

Let us decompose the event $\{\trim(\hat t_{\infty}) \text{ contains at least }q
\text{ repetitions}\}$ according to the different possible realisations of 
$\trim(\hat t_{\infty})$. We denote by $\sub(t_{\infty})$ 
the support of the random variable $\shape(\trim(\hat t_{\infty}))$ 
(where the randomness comes from the random labelling of the tree $t_{\infty}$). 
We let $\sub_m(t_\infty)$ be the subset of $\sub(t_\infty)$ 
consisting of the trees having~$m$ nodes that are leaves of~$t_\infty$. 
With the definition of size of a trimmed tree, 
the elements of $\sub_m(t_{\infty})$ all have size~$m$.
Given a tree $t\in \sub(t_\infty)$, we denote by $\Delta(t)$ (resp.\ 
$\Delta^{\sss [\geq q]}(t)$, resp.\ $\Delta^{\sss [0]}(t)$, resp.\ $\Delta^{\sss [<q]}(t)$) 
the set of all different leaf-labelled version of~$t$ (resp.\ having at least $q$ repetitions, resp.\ 
having no repetitions, resp.\ having at most $q-1$ repetitions); 
we see an element $\ell\in\Delta(t)$ as a function from the set of leaves of~$t$ 
(that are also leaves of $t_\infty$) to $\{x_1,\bar x_1,\dots, x_k, \bar x_k\}$.
Given a leaf $w$ of $t$, we denote by $\ell(w)$ its label according to $\ell\in\Delta(t)$.

Recall that $\Lab(w)$ denotes the random label of the leaf $w$ in $\hat t_\infty$. 
For $t\in \sub(t_\infty)$ and $\ell\in \Delta(t)$, we also 
write $\Lab(t:\hat t_{\infty})=\ell$ for the event 
``$\Lab(w)=\ell(w)$ for every leaf $w$ of $t_{\infty}$ that is also contained in $t$''.

The addition of an index $\mathtt{notrim}$ means that we restrict ourselves 
to the set of labellings $\ell$ of $t$, such that conditionally on $\Lab(t:\hat t_{\infty})=\ell$,
the trimming procedure leaves all the nodes of $t$ consistent, or in other words 
such that $\trim(t) = \ell$.

In the following we argue conditionally to the random labelling of the internal 
nodes of $\hat t_{\infty}$, and the reasoning is valid for any such labelling: 
we denote by $\bbP^{\star}$ this conditional probability. On the one hand, using 
Markov's inequality, we have, for any~$k$,
\begin{align}
&\bbP^{\star}(\trim(\hat t_{\infty}) \text{ contains at least }q\text{ repetitions})\notag\\
&= \bbP^{\star}(\|\trim(\hat t_{\infty})\|\geq \sqrt k ) 
+ \sum_{m=1}^{\lfloor\sqrt k\rfloor} \sum_{t\in \sub_m(t_{\infty})} 
\sum_{\ell \in \Delta_{\mathtt{notrim}}^{[\geq q]}(t)} 
\bbP^{\star}(\shape(\trim(\hat t_{\infty})) =  t \text{ and }\Lab(t:\hat t_{\infty})=\ell)\notag\\
&\leq \frac{\mathbb E^{\star}[\|\trim(\hat t_{\infty})\|^{2q}]}{k^q} 
+ \sum_{m=1}^{\lfloor\sqrt k\rfloor} \sum_{t\in \sub_m(t_{\infty})} 
\sum_{\ell \in \Delta_{\mathtt{notrim}}^{[\geq q]}(t)} 
\bbP^{\star}(\Lab(t:t_{\infty}) = \ell) \bbP^{\star}(\shape(\trim(\hat t_\infty))=t~|~ \Lab(t:\hat t_{\infty}) = \ell)\notag\\
&=\frac{\mathbb E[\|\trim(\hat t_{\infty})\|^{2q}]}{k^q}  
+ \sum_{m=q}^{\lfloor\sqrt k \rfloor} \frac1{(2k)^m} 
\sum_{t\in \sub_m(t_{\infty})} \sum_{\ell \in \Delta_{\mathtt{notrim}}^{[\geq q]}(t)}
\bbP^{\star}(\shape(\trim(\hat t_\infty))=t~|~ \Lab(t:\hat t_{\infty}) = \ell).\label{eq:proof(b)1}
\end{align}
Indeed, for every fixed $t\in \sub_m(t_\infty)$, and $\ell\in \Delta(t)$ the~$m$ leaves
of $\hat t_\infty$ that are in $t$ have independent labels, so that, for all $\ell\in\Delta(t)$,
\[\bbP^\star(\Lab(t:\hat t_{\infty})=\ell) = 
\bbP^\star(\Lab(w)=\ell(w) \text{ for all leaf }w\text{ of }t_{\infty}\text{ that is also a leaf of }t) 
= \frac 1 {(2k)^m}.
\]
We also used that $\mathbb E^{\star}[\|\trim(\hat t_{\infty})\|^{2q}] = \mathbb E[\|\trim(\hat t_{\infty})\|^{2q}]$
since the distribution of $\|\trim(\hat t_{\infty})\|$
does not depend on the labels of the internal nodes of $\hat t_{\infty}$.

Note that, on the other hand, we have, for all integer $m$,
\[\bbP^{\star}(\|\trim(\hat t_{\infty})\|=m)
= \frac1{(2k)^m} 
\sum_{t\in \sub_m(t_{\infty})} 
\sum_{\ell \in \Delta_{\mathtt{notrim}}(t)}
\bbP^{\star}(\shape(\trim(\hat t_\infty))=t~|~ \Lab(t:\hat t_{\infty}) = \ell).\]

To upper bound the number of repetitions, let us look at the following ratio for all $1\leq m <\sqrt k$,
\[\frac{\sum_{t\in \sub_m(t_{\infty})} 
\sum_{\ell \in \Delta_{\mathtt{notrim}}^{\sss [\geq q]}(t)}
\bbP^{\star}(\shape(\trim(\hat t_\infty))=t~|~ \Lab(t:\hat t_{\infty}) = \ell)}
{\sum_{t\in \sub_m(t_{\infty})} 
\sum_{\ell \in \Delta_{\mathtt{notrim}}(t)}
\bbP^{\star}(\shape(\trim(\hat t_\infty))=t~|~ \Lab(t:\hat t_{\infty}) = \ell)}
= \frac{1}{1+\mathtt{rat}}\leq \frac1{\mathtt{rat}},\]
where
\begin{align*}
\mathtt{rat}
&:= 
\frac{\sum_{t\in \sub_m(t_{\infty})} 
\sum_{\ell \in \Delta_{\mathtt{notrim}}^{\sss [<q]}(t)}
\bbP^{\star}(\shape(\trim(\hat t_\infty))=t~|~ \Lab(t:\hat t_{\infty}) = \ell)}
{\sum_{t\in \sub_m(t_{\infty})}
\sum_{\ell \in \Delta_{\mathtt{notrim}}^{\sss [\geq q]}(t)}
\bbP^\star(\shape(\trim(\hat t_\infty))=t~|~ \Lab(t:\hat t_{\infty}) = \ell)}\\
&\geq \frac{\sum_{t\in \mathtt{sub}_m(t_{\infty})} 
\sum_{\ell \in \Delta^{\sss [0]}(t)}
\bbP^{\star}(\shape(\trim(\hat t_\infty))=t~|~ \Lab(t:\hat t_{\infty}) = \ell)}
{\sum_{t\in \sub_m(t_{\infty})}
\sum_{\ell \in \Delta_{\mathtt{notrim}}^{\sss [\geq q]}(t)}
\bbP^{\star}(\shape(\trim(\hat t_\infty))=t~|~ \Lab(t:\hat t_{\infty}) = \ell)},
\end{align*}
because a labelled tree with no repetition cannot be trimmed, so for every $t$, and for every integer $q\geq 1$,
\[\Delta^{\sss [0]}(t)=\Delta^{\sss [0]}_\nt(t)\subseteq \Delta^{\sss [<q]}_\nt(t)
%\qquad \text{and}\qquad
%\lab^{\sss [\ge q]}_\nt(t)\subseteq \lab^{[\ge q]}(t).
\] 
%Fix a subtree $t$ of $t_{\infty}$, 
%and a labelling of this subtree $\ell\in \Delta^{[0]}(t)$.
Given $t\in\sub_m(t_{\infty})$, we denote by $\partial t$ the 
the set of nodes of $t$ that have no children in $t$, but that 
are not leaves of $t_\infty$; so $\partial t$ is the internal vertex boundary
of $t$ inside $t_\infty$. Each node of $\partial t$ has all its progeny that is 
inconsistent, but is not itself inconsistent (see Figure~\ref{fig:trim_explanations}).

\begin{figure}
\begin{center}
\includegraphics[width=.6\textwidth]{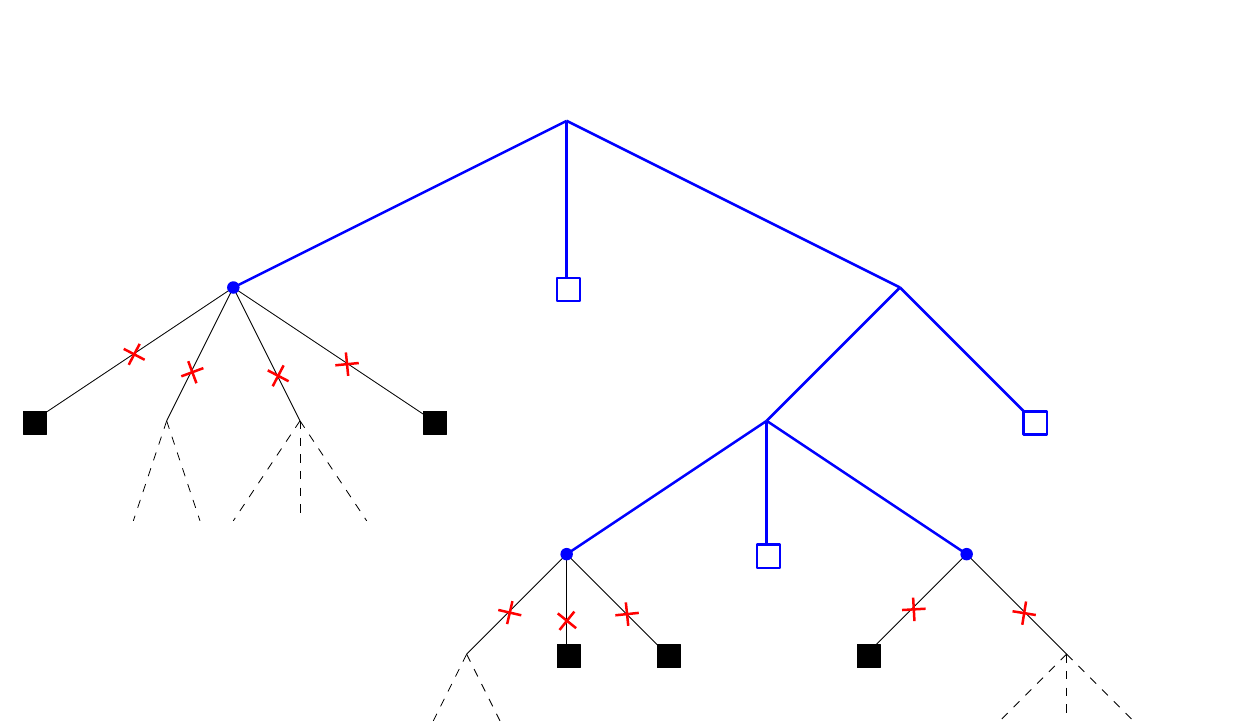}
\end{center}
\caption{This figure represents and infinite but locally finite tree $t_{\infty}$ 
(the dashed parts are possibly infinite trees that we don't represent in this picture) 
and one of its subtrees $t$ in bold blue.
To have $trim(\hat t_{\infty}) = t$ the nine cuts in red must have been made during the trimming procedure.
The three internal nodes marked by a black dot are the elements of $\partial t$, 
and if we denote by $v$ the leftmost one, $D(v) = 1$ and $F(v) = 2$.}
\label{fig:trim_explanations}
\end{figure}

Consider first the case of $\ell \in \Delta^{\sss [0]}(t)$. 
Given a node $v\in \partial t$, let us denote by $D(v)$ the number of leaves in $t$ 
whose father is an ancestor of $v$. 
Conditional on $\{\Lab(t)=\ell\}$, these $D(v)$ leaves 
are labelled according to $\ell$, and since $\ell$ contains no repetition,
they are labelled by $D(v)$ distinct variables.
These $D(v)$ leaves define a set of $D(v)$ literals such that: 
if at least one of the leaves whose parent is $v$ is labelled by one of these literals 
(or its negation, depending on the labelling of the internal nodes which we conditioned on), 
then all the children of $v$ are inconsistent (and thus cut by the trimming procedure).
Let us denote by $F(v)$ the number of leaves in $t_\infty$ whose parent is $v$; 
not that for $v\in \partial t$, none of these leaves can be in $t$.
Let also $p(d,f)$ be the probability that 
at least one among $f$ leaves is labelled by $\alpha_1, \alpha_2, \ldots, $ or $\alpha_d$ 
(where $\{\alpha_1, \ldots \alpha_d\}$ is any fixed subset of $d$ elements of 
$\{x_1, \bar x_1, \ldots, x_k, \bar x_k\}$), 
or two among those $f$ leaves are labelled by a literal and its negation.
(This is clearly independent of the set $\{\alpha_1,\alpha_2,\dots, \alpha_d\}$, provided 
it has cardinality $d$.)
Then, the arguments above and the independence of the leaf labels in $\hat t_\infty$ imply 
that given a subtree $t$ and a labelling of this subtree $\ell\in \Delta^{\sss [0]}(t)$, 
conditioned on the labels of the internal nodes of $\hat t_{\infty}$,
\[\bbP^{\star}(\shape(\trim(\hat t_\infty))=t~|~ \Lab(t:\hat t_{\infty}) = \ell) 
= \prod_{v \in \partial t} p(D(v), F(v)).\]
Note that this probability does not depend on the labels of the internal nodes of 
$\hat t_{\infty}$. This is due to the fact that in the trimming procedure, the labels 
$\land$ and $\lor$, literals and their negations behave symmetrically:
the constraint $\alpha=\True$ is generated by a leaf labelled by $\alpha$ whose parent 
is labelled by $\land$ as well as by a leaf labelled by $\bar \alpha$ whose parent is 
labelled by $\lor$. 

The same arguments are valid for every subtree $t$ of $t_{\infty}$, and for every 
$\ell\in \Delta^{\sss [\geq q]}(t)$, except that since there may be some repetitions in 
$\ell$, for every leaf $v$ in $\partial t$, the number of labellings permitting to trim is 
at most $D(v)$ (and not exactly $D(v)$ as above). 
Therefore, since the function $p(d,f)$ is increasing in $d$,
for every $t\in t_{\infty}$, and $\ell\in \Delta^{\sss [\geq q]}(t)$:
\begin{equation}\label{eq:proof(b)ineq}
\bbP^{\star}(\shape(\trim(\hat t_\infty))=t~|~ \Lab(t:\hat t_{\infty}) = \ell) 
\leq \prod_{v \in \partial t} p(D(v), F(v)).
\end{equation}

Thus, since the cardinalities $|\Delta^{\sss [\geq q]}(t)|$ and $|\Delta^{\sss [0]}(t)|$ depend on the size $m$ of $t$ but not on $t$ itself, we have
\begin{align*}
\mathtt{rat} 
&\geq \frac{\sum_{t\in \sub_m(t_{\infty})} 
\prod_{v \in \partial t} p(D(v), F(v))\,|\Delta^{\sss [0]}(t)|}
{\sum_{t\in \sub_m(t_{\infty})} 
\prod_{v \in \partial t} p(D(v), F(v)) \, |\Delta^{\sss [\geq q]}_{\nt}(t)|}\\
&\geq \frac{\sum_{t\in \sub_m(t_{\infty})} 
\prod_{v \in \partial t} p(D(v), F(v)) \, |\Delta^{\sss [0]}(t)|}
{\sum_{t\in \sub_m(t_{\infty})} 
\prod_{v \in \partial t} p(D(v), F(v)) \, |\Delta^{\sss [\geq q]}(t)|}
\geq \frac{|\Delta_m^{\sss [0]}|}{|\Delta_m^{\sss [\geq q]}|},
\end{align*}
since, for any integer~$q$, for any subtree~$t$ of~$t_{\infty}$, $|\Delta^{\sss [\geq q]}_{\nt}(t)|\leq |\Delta^{\sss [\geq q]}(t)|$, 
where $|\Delta_m^{\sss [\geq q]}|$ (resp.\ $|\Delta_m^{\sss [0]}|$) is the number of ways to label $m$ 
leaves with at least~$q$ repetitions (resp.\ no repetition).
We have,
\[|\Delta_m^{\sss [0]}|
= k(k-1)\cdots (k-m+1) = \frac{k!}{(k-m)!},\]
and,
\[|\Delta_m^{\sss [\geq q]}|
= \sum_{v=1}^{m-q} \left\{\begin{matrix}m\\v\end{matrix}\right\} k (k-1) \cdots (k-v+1)
\leq m^q \sum_{v=1}^{m-q} \left\{\begin{matrix}m-q\\v\end{matrix}\right\} \frac{k!}{(k-v)!}
\leq m^q k^{m-q},\]
where for all integers $n, p\ge 1$, ${n \brace p}$ is the Stirling number of second kind, i.e.
the number of ways to partition a set of $n$ elements into $p$ non-empty parts. 
We have used the standard equality (see, e.g., \cite[][p.\ 207]{Comtet1974a})
\[\sum_{p=1}^n \left\{\begin{matrix}n\\p\end{matrix}\right\} x(x-1)\cdots (x-p+1) = x^n,\]
for all integers $n$ and $x$.
Thus,
\[\mathtt{rat}\geq \frac{k(k-1)\cdots (k-m+1)}{m^q k^{m-q}}\geq \frac{k^q}{m^q} \left(1-\frac{m}{k}\right)^m
\geq \frac{k^q}{2\mathtt e m^q},\]
for all $m < \sqrt k$, for all $k\geq k_0$, where $k_0$ is defined as
$k_0 = \max\{k\geq 1 \colon (1-\nicefrac1{\sqrt k})^{\sqrt k} < \nicefrac1{2\mathtt e}\}$.
We thus get
\begin{align*}
\bbP(\trim(\hat t_{\infty}) \text{ contains at least }q\text{ repetitions})
&\leq \frac{\mathbb E\|\trim(\hat T_{\infty})\|^{2q}}{k^q} + 
2\mathtt e \sum_{m=1}^{\lfloor \sqrt k\rfloor} \frac{m^{q}}{k^q}
\mathbb P(\|\trim(\hat t_{\infty})\|=m)\\
&=\frac{\mathbb E\|\trim(\hat t_{\infty})\|^{2q}+ 2\mathtt e\ \mathbb E\|\trim(\hat t_{\infty})\|^{q}}{k^q},
\end{align*}
for all $k\geq k_0$, as desired.

\vspace{\baselineskip}
$(b)$ We follow the line of argument and use the same notations as in the proof 
of statement $\bds{(a)}$. We have (see Equation~\eqref{eq:proof(b)1})
\begin{align*}
&\mathbb P^{\star}(\trim(\hat t_{\infty})\text{ has size $p$ and at least $q$ repetitions})\\
&\hspace{2cm}= \frac1{(2k)^p} 
\sum_{t\in \sub_p(t_{\infty})} \sum_{\ell \in \Delta_{\mathtt{notrim}}^{[\geq q]}(t)}
\bbP^{\star}(\shape(\trim(\hat t_\infty))=t~|~ \Lab(t:\hat t_{\infty}) = \ell),
\end{align*}
and (see Equation~\eqref{eq:proof(b)ineq}),
\[\bbP^{\star}(\shape(\trim(\hat t_\infty))=t~|~ \Lab(t:\hat t_{\infty}) = \ell) 
\leq \prod_{v \in \partial t} p(D(v), F(v)).\]
Recall that $p(d,f)$ is, by definition, the probability that either
at least one among $f$ leaves is labelled by one of $\alpha_1, \alpha_2, \ldots,$ $\alpha_d$, 
or two among those $f$ leaves are labelled by a literal and its negation.
Using the bound $(1-x)^n\geq 1-nx$ for all integer $n\geq 1$ and for all $x\in (0,1)$, we 
obtain
\[1-p(d,f) = \frac{(2k-d)!}{(2k-d-f)!(2k)^f}\geq 
\left(1-\frac{d+f}{2k}\right)^f \geq 1-\frac{(d+f)f}{2k},\]
so that
\[p(d,f)\leq \frac{(d+f)f}{2k}.\]
Note that since $\hat t_{\infty}$ is an infinite tree and $t\in \sub_p(t_{\infty})$ has 
size $p$, the set $\partial t$ is non empty. 
Let $v$ be a node of $\partial t$: this node has height at most $p$ because 
the tree $t_{\infty}$ and thus $t$ have no unary node (by definition of an and/or tree,
see Section~\ref{sub:def}).
Since $t_{\infty}$ is locally finite, the constant 
\begin{equation}\label{eq:def_F}
F_{\mathtt{max}}:=\max\{F(v) \colon v\in t_{\infty} \text{ having  height at most }p\}
\end{equation}
is finite,
and a similar argument implies that the constant
\begin{equation}\label{eq:def_A}
D_{\mathtt{max}}:=\max\{D(v) \colon v\in t_{\infty} \text{ having  height at most }p\}
\end{equation}
is also finite,
so that that for all node $v\in \partial t$
\[p(D(v), F(v))\leq \frac{(D_{\mathtt{max}}+F_{\mathtt{max}})F_{\mathtt{max}}}{2k},\]
and thus,
\[\bbP^{\star}(\trim(\hat t_\infty)=\ell~|~ \Lab(t) = \ell) 
\leq \frac{(D_{\mathtt{max}}+F_{\mathtt{max}})F_{\mathtt{max}}}{2k}.\]
Therefore, we get
\begin{align*}
&\mathbb P^{\star}(\trim(\hat t_{\infty})\text{ has size $p$ and at least $q$ repetitions})\leq \frac{(D_{\mathtt{max}}+F_{\mathtt{max}})F_{\mathtt{max}}}{(2k)^{p+1}}\ 
|\sub_p(t_{\infty})|\ |\Delta^{\sss [\geq q]}_p|.
\end{align*}
Recall that
\[|\Delta_p^{\sss [\geq q]}|
\leq p^q \frac{k!}{(k-p+q)!}\leq p^q k^{p-q},\]
which gives
\[\mathbb P^{\star}(\trim(\hat t_{\infty})\text{ has size $p$ and at least $q$ repetitions})
\leq (D_{\mathtt{max}}+F_{\mathtt{max}})F_{\mathtt{max}}\ |\sub_p(t_{\infty})|
\ \frac{p^q}{2^{p+1} k^{q+1}}.\]
Noting that $|\sub_p(t_{\infty})|$ is finite concludes the proof with the following choice of $K_{p,q}(t_{\infty})$
%\makeatletter\tagsleft@true\makeatother

\begin{equation}\label{eq:K}
K_{p,q}(t_{\infty}) 
:= (D_{\mathtt{max}}+F_{\mathtt{max}})F_{\mathtt{max}}\ |\sub_p(t_{\infty})|
\ \frac{p^q}{2^{p+1}},
\end{equation}
which completes the proof of $(b)$.
%\makeatletter\tagsleft@false\makeatother
\end{proof}

\subsection{Probability of the two constant functions}
Let $T_{\infty}$ be the limit tree of the random sequence $(T_n)_{n\geq 0}$, and suppose that 
it has a unique end. Let $u_0, u_1, \ldots $ be the nodes of the spine, starting from the root.
Recall that, as already mentioned, the distribution $(T_n)_{n\geq 0}$ and thus $T_{\infty}$ may depend on $k$
(we will provide such an example at the end of the section).
For all $i\geq 0$, we denote by $A_i$ the random forest of finite subtrees hanging from the node $u_i$.

%We consider yet another way of cutting the spine of $\hat{T}_{\infty}$: 
%we cut the infinite spine \emph{as soon as} the obtained tree after such a cut computes $f[\hat{T}_{\infty}]$. 
%Note that this cut happens no later than the one described in the proof of Theorem~\ref{thm:cv}, 
%thus it happens at an almost surely finite depth.
%We denote by $\simp(\hat T_{\infty})$ the tree obtained after 
%(1) cutting the spine of $\hat T_{\infty}$ with this new procedure
%and (2) trimming each of the forests $(\hat A_i)_{i\geq 0}$ 
%(with the trimming procedure of Section~\ref{sec:improved_bounds}).
%Note that $\simp(\hat T_{\infty})$ is contained in $\trim(\hat T_{\infty})$:
%\begin{lem}\label{lem:simp<trim}
%For all (finite or infinite) Boolean tree $t$, $\simp(\hat t) \subseteq \trim(\hat t)$.
%\end{lem}

Let us denote by $L_i$ the number of leaves of $\trim(\hat A_i)$. 
In the following, we will assume that
\begin{enumerate}[(\tt{H})]
\item  the sequence $(A_i)_{i\geq 0}$ is i.i.d.
\end{enumerate}
When ({\tt H}) holds, we will denote by $A$ a random variable having the common law 
of the $A_i$, and by $\trim(\hat A)$ the trimmed version of its random labelling. 
Observe that ({\tt H}) implies that $(L_i)_{i\geq 0}$  is a sequence of i.i.d. random 
variables as well, and we denote by $L$ a random variable having their common law.

Given a forest $a$, and two of its leaves $\ell_1$ and $\ell_2$ we denote 
by $C_{\ell_1, \ell_2}$ the number of nodes (of out-degree at least two) of the union of the path 
from $\ell_1$ to one root of $a$ and the path from $\ell_2$ to one root (possibly the same) of $a$. 
We denote by $C_a$ the minimum of those $C_{\ell_1, \ell_2}$ taken on all couple of leaves $(\ell_1,\ell_2)$ 
and by $N_a$ the number of such couples that realize this minimum.

%We say that a sequence $(g_n)_{n\ge 1}$ is {\bf regular at infinity} 
%if either $g_n\to\infty$ or 
%$\limsup_{n\to\infty}~g_n<~\infty$.

\begin{lem}\label{lem:theta_true}
Suppose that the hypothesis of Theorem~\ref{thm:cv} holds, that $T_\infty$ has almost surely 
a unique end and satisfies ({\tt H}).
%Suppose further that $(\Ec{L^2}/k)_{k\ge 1}$ is regular at infinity.
Then, asymptotically as $k$ tends to infinity,
\[\frac{1}{k}\; \E{\frac{N_A}{2^{C_A}}} 
\leq \bbP_k(\True) 
= \bbP_k(\False) 
\leq 
\frac{(2\mathtt e+1)\mathbb E[L]+ \mathbb E[L^2]}{k}.
\]
%1-\mathtt{e}^{-\nicefrac{\mathbb E L^2}{k}} + \frac{\mathbb E L}{k}.
\end{lem}

\medskip
\noindent{\bf Example.}
Before proceeding to the proof, let us present an example that shows that 
the moment conditions in Lemma~\ref{lem:theta_true} cannot be removed altogether. 
Consider a sequence $(X_i)_{i\ge 0}$ of
i.i.d.\ copies of an integer-valued random variable $X$. Let $T_\infty$ consists of an 
infinite spine $(u_i)_{i\ge 0}$, and such that, for each $i\ge 0$, the node $u_i$ has $X_i$ 
leaf-children aside from $u_{i+1}$. Then $L=X$, $C=0$ and $N=X(X-1)/2$. Then,
Lemma~\ref{lem:theta_true} implies that, if $\Ec{X^2}=o(k)$ as $k\to\infty$, we have 
\[\frac{\Ec{X^2}-\Ec{X}}{2k} \le \bbP_k(\True)= \bbP_k(\False) \le \frac{\Ec{X^2} + (2\mathtt e+1)\Ec{X}}k,\]
implying that $\bbP_k(\True) = \bbP_k(\False) = o(1)$;
while if $\Ec{X^2}= \Theta(k)$ then $\bbP_k(\True)=\bbP_k(\False)=\Theta(1).$

\begin{cor}\label{cor:theta_true_particular_case}
Suppose that the assumptions of Lemma~\ref{lem:theta_true} are satisfied.
If additionally, we assume that $\limsup_{k\to+\infty} \mathbb E L^2 < +\infty$, then
asymptotically as $k$ tends to infinity,
$$\bbP_k(\True) = \bbP_k(\False) = \Theta(\nicefrac1k).$$
\end{cor}

\begin{proof}
The upper bound is straightforward from Lemma~\ref{lem:theta_true} and from the additional hypothesis. 
For the lower bound observe that $\limsup_{k\to+\infty} \mathbb E L^2 <+\infty$ 
implies that there exists an integer $K>0$ such that the probability that $L\leq K$ 
is greater than a positive constant $c$ (note that these two constants do not depend on $k$). 
%{\nic Observe that the nodes with arity one have no influence on the function computed. }
% I think I forbade node of arity one in the introduction
Remark as well that, since $\mathtt{trim}(\hat A)$ is smaller than $A$, 
$N_A \geq N_{\mathtt{trim}(\hat A)}$ and
$C_A \leq C_{\mathtt{trim}(\hat A)}$, which implies
\[\frac1k\;\E{\frac{N_A}{2^{C_A}}}\geq 
\frac1k\;\E{\frac{N_{\mathtt{trim}(\hat A)}}{2^{C_{\mathtt{trim}(\hat A)}}}}\geq \frac{c}{k2^K} = \Omega(\nicefrac1k).\qedhere\]
\end{proof}

\begin{proof}[Proof of Lemma~\ref{lem:theta_true}]
The constant function $\True$ is represented by any tree that has two leaves labelled by some 
variable and its negation both connected to the root by a path of $\lor$ connectives.
We thus have
\[\bbP_k(\True)\geq \frac1k \sum_{a\in\Supp A} \mathbb{P}(A=a) \cdot \frac{N_A}{2^{C_A}} 
= \frac{1}{k}\; \E{\frac{N_A}{2^{C_A}}}.\]

Let us now prove the upper bound. 
%First observe that if $\Ec{L^2}/k \to+\infty$ when $k$ 
%tends to infinity, then the upper bound in the statement is trivial. 
%By regularity at infinity, we may now assume that $\limsup_{k\to+\infty}\Ec{L^2}/k <+\infty$.
Assuming that $\hat T_{\infty}$ calculates $\True$ implies that
\begin{itemize}
\item if the root is $\land$-labelled, then the conjunction of the subtrees af $\hat A_0$ calculates $\True$;
\item if the root is $\lor$-labelled,
\begin{itemize}
\item either the disjunction of the subtrees of $\hat A_0$ calculates $\True$;
\item or the infinite subtree of the root of $\hat T_{\infty}$ calculates $\True$;
\item or the disjunction of the subtrees of $\hat A_0$ calculates a non-constant function $\eta$, the infinite subtree of the root calculates a non-constant function $f$, and $\eta \lor f$ is the constant function $\True$: in this case, we say that there is {\it compensation}.
\end{itemize}
\end{itemize}
Assumption $\mathtt{(H)}$ tells us that the infinite subtree of the root of $\hat T_{\infty}$ is distributed as $\hat T_{\infty}$. Thus, the above implication tells us that
\begin{align*}
\mathbb P(\of[\hat T_{\infty}]= \True)
\leq &\frac12 \mathbb P(\of[\hat A_{\lor}]= \True) 
+ \frac12 \mathbb P(\of[\hat T_{\infty}]= \True)\\
&+ \frac12 \mathbb P(\of[\hat A_{\land}]= \True)
+ \frac12 \mathbb P(\text{compensation}),
\end{align*}
where $\of[\hat A_{\lor}]$ (resp. $\of[\hat A_{\land}]$) stands for the Boolean function calculated by the disjunction (resp. conjunction) of the subtrees of the forest~$\hat A$.
Using the fact that $\mathbb P(\of[\hat A_{\land}]= \True)\leq \mathbb P(\of[\hat A_{\lor}]= \True)$, 
we obtain
\begin{equation}\label{eq:decomposition}
\mathbb P(\of[\hat T_{\infty}]= \True)
\leq 2\mathbb P(\of[\hat A_{\lor}]= \True) + \mathbb P(\text{compensation}).
\end{equation}

Note that the event $\{\of[\hat A_{\lor}]= \True\} = \{\of[\trim (\hat A)_{\lor}]= \True\}$ is contained in the event 
``at least two leaves of $\trim(\hat A)$ are labelled by the same variable'':
\[\mathbb P(\of[\hat A_\vee]= \True)
\leq \mathbb P(\trim(\hat A)\text{ contains at least 1 repetition}),\]
which implies, in view of Lemma~\ref{lem:trim}:
\begin{equation}\label{eq:gi}
\p{\of[\hat A_\vee]= \True}
\leq \frac{2\mathtt e \Ec{\|\trim (\hat A)\|} + \Ec{\|\trim (\hat A)\|^2}}{k} \leq \frac{2\mathtt e \mathbb E L + \mathbb E L^2}{k}.
\end{equation}

Let us now study the probability of the event $\{\text{compensation}\}$.
On this event, the disjunction of the subtrees of $\hat A_1$ calculates a non-constant function, 
which we denote by $\eta$. Note that $\eta$ has then at least one essential variable.
Let us denote by $f$ the random Boolean function calculated by the infinite subtree (rooted at $\nu_2$).
Let us prove that the event $\{\eta \lor f\equiv \True\}$ is contained into the event 
``at least one essential variable of $\eta$ is an essential variable for $f$''.
Assume for a contradiction that no essential variable of $\eta$ is essential for $f$: 
then there exists an assignation $w$ of the essential variables of $\eta$ such that
$\eta_{|w} = \False$ and $f_{|w}\neq \True$ thus, $f\lor\eta_{|w} \neq \True$,
which is impossible since $f\lor\eta_{|w}=\True$ by assumption. 
Thus there exists at least one essential variable of $\eta$ which is also essential for $f$.
Let us denote by $\Gamma$ the (random) number of essential variables of $\eta$. 
Note that, with probability one, $\Gamma\le L_1$, and that $L_1$ is (unconditioned and) 
distributed as $L$. Then, 
\begin{align}\label{eq:gic}
\pc{f\lor \eta = \True \text{ and } \text{compensation}}
&\leq \sum_{\gamma=1}^{k} \pc{\Gamma=\gamma}
\sum_{j=1}^{\gamma}\pc{x_j\text{ is an essential variable of }\eta} \notag \\
&\leq \sum_{\gamma=1}^{k} \pc{\Gamma=\gamma}\frac{\gamma}{k}\notag\\
&\leq \frac{\mathbb E L}{k}.
\end{align}
Combining Equations \eqref{eq:gi}, \eqref{eq:gic} and~\eqref{eq:decomposition} completes the proof.
\end{proof}

\noindent{\bf Example: the associative trees.} 
We already introduced this example in Section~\ref{sub:assoc}:
Let $G_{n,k}$ be a random tree uniformly distributed among trees having $n$ nodes (leaves and internal nodes), 
labelled with $k$ variables and such that no internal node has a unique child. 
Forget the labels of $G_{n,k}$, it gives a random non-labelled tree $T_{n,k}$ such that 
$\hat T_{n,k}$ and $G_{n,k}$ have the same distribution. 
In this case, $T_{n,k}$ really depends on $k$ as so does the random variable $A=A_k$
(being the common law of the random forests of trees hanging on the infinite spine).
We have shown in Subsection~\ref{sec:examples} that $T_{n,k}$ is a critical Galton--Watson tree.
 Using Equation~\eqref{eq:GW_assoc}, one can check that, with high probability when $k\to+\infty$,
$C_A = 0$, and that $\Ec{N_A 2^{-C_A}} \sim k$ as $k$ tends to infinity.
Lemma~\ref{lem:theta_true} applies and gives that $\bbP_k(\True) = \Theta(1)$ (a result already proved in~\cite{GGKM15}).

\subsection{Probability of a given Boolean function}
In this section, we aim at proving the analog of the result of \citet{Kozik2008a}, 
namely estimate $\mathbb P_k(f)$ for any given Boolean function $f$.
To do so, we will need additional assumptions: we still assume~({\tt H}) and additionally require that
$A$ and $L$ do \emph{not} depend on $k$ (recall that the local limit may depend on $k$).

%We will see later on an example (namely the already mentioned associative trees) 
%how the assumption that $A$ and $L$ do not depend on $k$ could be removed for some particular cases.

Given a Boolean function $f$, and a random forest $\Upsilon$, 
we denote by $L_{\Upsilon}(f)$ the effective complexity of $f$ according to $\Upsilon$ as follows.
Let $\mathcal M_{\Upsilon}(f)$ be the set of forests in the support of $\Upsilon$ that can be 
labelled so that the disjunction of their subtrees calculates the function $f$; 
$L_{\Upsilon}(f)$ is the the size of the smallest forests in $\mathcal M_{\Upsilon}(f)$.
In the following theorem, the quantity $L_A(f)$ is used; 
we recall that, under ({\tt H}), $A$ is a random forest whose distribution is the joint distribution of the~$A_i$'s, 
the forests hanging on the spine of the local limit of $(T_{k,n})_{n\geq 1}$.

For any tree $t$ and any integer $m$, we denote by $N_m(t)$ the number of nodes that are at distance at most $m$ of the root in $t$.
 
\begin{thm}\label{thm:theta}
Suppose that for $k\ge 1$, $(T_{k,n})_{n\ge 0}$ is a sequence of unlabelled random trees converging in 
distribution to a local limit $T_\infty$ with a unique end, satisfying $\mathtt{(H)}$, 
and such that $A$ and $L$ do not depend on~$k$. 
Suppose further that for all integer $m\geq 1$,
\begin{enumerate}[(i)]
\item $\Ec {N_m^{m+2}(T_{\infty})}<+\infty$, and 
\item $\Ec{\|\trim(\hat T_\infty)\|^m}<\infty$.
\end{enumerate}
Let $k_0$ be an integer independent of $k$ 
and let $f$ be a $k_0$-variable Boolean function, then there 
exists constants $c,C\in (0,\infty)$ such that for all $k$ large enough
\[c \cdot \frac{\bbP_k(\True)}{k^{L_A(f)}}
\leq \bbP_k(f) 
\leq C \cdot \frac{\bbP_k(\True)+\nicefrac1k}{k^{L(f)}}.\]
\end{thm}

Note that Assumption~$(ii)$ implies in particular that~$L$ has finite moments of all orders, and thus, under the assumptions of Theorem~\ref{thm:theta}, $\mathbb P_k(\True) = \Theta(\nicefrac1k)$, in view of Corollary~\ref{cor:theta_true_particular_case}.
Also note that in all examples considered in this article, 
the support of the random variable $A$ will be such that $L_A\equiv L$ on the set of Boolean function, 
and the following Corollary applies to these examples:

\begin{cor}
Under the assumptions of Theorem~\ref{thm:theta} and assuming that $L_A(f) = L(f)$, we have
\[\bbP_k(f) = \Theta\left(\frac1{k^{L(f)+1}}\right).\]
\end{cor}

\begin{proof}[Proof of Theorem~\ref{thm:theta}]
If $\hat T_{\infty}$ is $\lor$-rooted, if
the infinite subtree rooted at $u_1$ calculates $\False$
(recall that $u_1$ is the second node on the infinite spine, $u_0$ being the root of $\hat T_{\infty}$), 
and if the disjunction of the subtrees of $A_0$ calculates $f$,
then $\hat T_{\infty}$ calculates $f$.
Let $t$ be a forest in $\mathcal M_A(f)$. Since the number of internal 
nodes is maximized when $t$ is a binary tree, $\frac1{k^{\|t\|}} \frac1{2^{\|t\|-1}}$ is a 
lower bound of the probability that $t$ is labelled such that $\hat t$ calculates $f$, 
and the probability that $\hat A_0$ calculates $f$ is at least
$\bbP(A=t) \frac1{k^{\|t\|}} \frac1{2^{\|t\|-1}}$. As the
labels of disjoints sets of nodes are independent, the lower bound is thus proved. 

We now focus on the upper bound. If $\hat{T}_{\infty}$ represents $f$, then 
$\trim(\hat T_{\infty})$ also computes the function $f$, which implies that its size 
is at least $L(f)$. Thus in order to prove the upper bound, it suffices to show that
\[
\pc{\of[\hat T_{\infty}]=f \text{ and } \|\trim(\hat T_\infty)\|\ge L(f)} 
= \mathcal{O}\left(\frac{1}{k^{L(f)+1}}\right).
\]
%Let us denote by $j$ the level where the spine has been cut (meaning that we have cut the spine between node $u_{j}$ and node $u_{j+1}$), 
%and let $\diamond_0,\diamond_1,\ldots$ 
%be the connectors along the spine of $\hat{T}_{\infty}$, from top to bottom
%(i.e. the respective labels of $u_0, u_1, \ldots$).
%If we denote by $f_0^{(i)}, \ldots, f_{r_i}^{(i)}$ the Boolean functions computed by 
%the finite sub-trees attached at level $i$ to the infinite spine and if we denote by
%$\bff_i = f_1^{(i)}\diamond_i \ldots \diamond_i f_{r_i}^{(i)}$, then
%\[f = \bff_1 \diamond_1 (\bff_2 \diamond_2 \ldots (\bff_{j-1} \diamond_{j-1} \bff_{j})),\]
%moreover, for all $i<j$,
%\[f \neq \bff_0 \diamond_0 (\bff_2 \diamond_2 \ldots (\bff_{i-1} \diamond_{i-1} \bff_i)).\]
%Let us denote by $\eta$ the Boolean function calculated by the tree that has been chopped off by cutting the spine between $u_{j}$ and $u_{j+1}$:
%\[f = \bff_0 \diamond_0 (\bff_1 \diamond_1 \ldots (\bff_{j-1} \diamond_{j-1} (\bff_{j}\diamond_{j} \eta))).\]
%We now distinguish cases depending on the size of $\trim(\hat T_\infty)$. 
%First note that
%\[\mathbb P(f[\hat T_{\infty}]=f) = 
%\sum_{t\in \mathcal B} \mathbb P(f[\hat T_{\infty}]=f \text{ and } \simp(\hat T_{\infty}) = t),\]
%where $\mathcal B$ denote the set of all and/or trees.
%
%We now distinguish cases depending on the size of $\simp(\hat T_\infty)$, 
%and denote by $\mathcal B_m$ the set of all and/or trees of size $m$, for all integer $m$.
Let us first prove that
\begin{equation}\label{eq:trim+=f}
\bbP(\of[\hat T_{\infty}]=f \text{ and } \|\trim(\hat T_{\infty})\| = L(f)) = \mathcal O \left(\frac{1}{k^{L(f)+1}}\right).
\end{equation}
Recall that the number of repetitions is formally defined as the difference 
between the number of leaves and the number of pairwise different variables appearing as labels of these leaves 
(in their positive or negated form).
Assume that $\trim (\hat T_{\infty})$ has size $L(f)$ and 
that $d\geq 1$ leaves of $\trim (\hat T_{\infty})$ are labelled by a non-essential 
variable (or its negation). 
Assign all the non-essential variables appearing in $\trim(\hat T_{\infty})$ to $\True$, 
and simplify the tree according to Boolean logic, i.e. using the following four simplifying 
rules: for every Boolean function $h$
\[\True \land h = h \qquad \True \lor h = \True 
\qquad \False \land h = \False \qquad \False \lor h = h.\]
Since we have only assigned values to non-essential variables, 
the tree obtained still calculates $f$, but since we have removed $d$ leaves, 
it has size at most $L(f)-d\leq L(f)-1$, which is impossible since $L(f)$ is 
minimum size of a tree computing $f$.
Therefore, if $\trim (\hat T_{\infty})$ has size $L(f)$, then all its leaves are labelled 
by essential variables, which implies that it contains exactly $L(f)-\gamma$ repetitions,
where $\gamma$ denotes the number of essential variables of $f$.

%Consider $\nu$ the random node defined as follows: $\trim(\hat T_{\infty})$ 
%defines a set of internal nodes $\mathcal C(\trim(\hat T_{\infty}))$ (see in Section~\ref{sec:repetitions}).
%The node $\nu$ is defined as the leftmost node of $\mathcal C(\trim(\hat T_{\infty}))$.
%From the fact that all children of $\nu$ are inconsistent while $\nu$ is not, we deduce that the tree consisting 
%of the union of $\trim(\hat T_{\infty})$ and the $F(\nu)$ leaves whose parent is $\nu$ contains at least $L(f)-\gamma+1$ repetitions:
%either one of the leaves whose parent is $\nu$ is labelled by one of the essential variables (already appearing in $\trim(\hat T_{\infty})$),
%or there is a literal and its negation appearing as labels of two leaves whose parent is $\nu$ (or both).
In other words,
\begin{align*}
& \{\of[\hat T_{\infty}]=f \text{ and }\|\trim(\hat T_{\infty})\|=L(f)\}\\ 
&\hspace{2cm}
~~ \subseteq ~~
\{\trim (\hat T_{\infty}) \text{ has size $L(f)$ and contains at least }L(f)-\gamma\text{ repetitions}\}.
\end{align*}
Note that by symmetry, the above inclusion is true for any Boolean function $f$ having 
$\gamma$ essential variables. It thus implies that,
\begin{align*}
&\binom{k}{\gamma}
\bbP(\of[\hat T_{\infty}]=f\text{ and } \|\trim(\hat T_{\infty})\| = L(f))\\
&\hspace{1cm}\leq \mathbb P(\trim (\hat T_{\infty}) \text{ has size $L(f)$ and contains exactly }L(f)-\gamma\text{ repetitions})\\
&\hspace{1cm}\leq \frac{\Ec{K_{L(f), L(f)-\gamma}(T_{\infty})}}{k^{L(f)+1-\gamma}},
\end{align*}
in view of Lemma~\ref{lem:trim}$(b)$,
where $K_{p,q}(t_{\infty})$ is defined in Equation~\eqref{eq:K} for any infinite locally finite tree $t_{\infty}$ and any integers $p$ and $q$.
Since $\binom k \gamma \ge (k-\gamma)^\gamma$, 
it is enough to prove that $\Ec{K_{L(f), L(f)-\gamma}(T_{\infty})}$ is a finite constant (independent of~$k$).
Note that $|\mathtt{sub}_{p}(T_{\infty})|$ is bounded from above by 
$\binom{N_{p}(T_{\infty})}{p}$ for all integer $p$ 
where $N_{p}(T_{\infty})$ is the (random) number of nodes at height at most $p$ in 
$T_{\infty}$.
Moreover, in view of the definitions of $A_{\mathtt{max}}$ and $F_{\mathtt{max}}$ (see Equations~\eqref{eq:def_F} and~\eqref{eq:def_A}),
we have that $F_{\mathtt{max}}$ is bounded from above by the maximal out-degree of all nodes at height at most $L(f)$ in $T_{\infty}$,
and thus $F_{\mathtt{max}}\le N_{\scriptscriptstyle L(f)+1}(T_{\infty})$.
Note also that $A_{\mathtt{max}}\le L(f)$. It follows that
\[K_{\sss L(f), L(f)-\gamma}(T_{\infty})
\leq (L(f)+N_{\sss L(f)+1}(T_{\infty}))N_{\sss L(f)+1}(T_{\infty})
\ \binom{N_{\sss L(f)}(T_{\infty})}{L(f)}
\ \frac{L(f)^{L(f)-\gamma}}{2^{L(f)+1}},\]
and thus has finite expectation in view of Assumption $(i)$.
We thus have proved \eqref{eq:trim+=f}

The second case, when $\|\trim(\hat T_\infty)\|>L(f)$ is similar, but simpler. 
Let us now prove that
\begin{equation}\label{eq:trim+>=f}
\mathbb P(\of[\hat T_{\infty}]=f \text{ and } \|\trim(\hat T_{\infty})\| \geq L(f)+1) 
= \mathcal O \left(\frac{1}{k^{L(f)+1}}\right)
= \mathcal O\left(\frac{\mathbb P_k(\True)+\nicefrac1k}{k^{L(f)}}\right).
\end{equation}
Assume that $\trim(\hat{T}_{\infty})$ has $d\geq p-L(f)$ leaves labelled by variables that are 
non-essential for $f$. Take $\ell$ the left-most one, and denote by $\nu$ its closer ancestor having arity at least 2.
Assign $\ell$ to $\True$ if its $\nu$ is labelled by $\lor$ 
or to $\False$ if $\nu$ is labelled by $\land$. This permit to assign $\nu$ to $\True$ (resp. $\False$) and
to cut all its children (among which there is at least one other leaf since $\nu$ has arity at least 2). 
Assign all other non-essential variables appearing in $\trim(\hat T_{\infty})$
to $\True$ (this assigns some of the leaves to $\True$, and others to $\False$ depending on the polarity of the literal labelling them). 
This operation does not change the function calculated by the tree and after simplification, 
the obtained tree has size at most $p-(d+1)\leq L(f)-1$, which is impossible.

Thus, $\trim(\hat{T}_{\infty})$ contains at least $L(f)+1$ leaves labelled by essential 
variables of $f$, and since $f$ has $\gamma$ essential variables, 
it contains at least $L(f)+1-\gamma$ repetitions.
It follows that
\begin{align*}
&\{\|\trim(\hat T_\infty)\|\ge L(f)+1 \text{~and~} \of[\hat T_{\infty}]=f\} \\
&\hspace{2cm}\subseteq \{\trim(\hat T_{\infty})\text{ contains at least }L(f)+1-\gamma\text{ repetitions}\}.
\end{align*}
The above inclusion is true for any Boolean function having $\gamma$ essential variables, which implies that
\begin{align*}
&\binom{k}{\gamma}
\mathbb P \left(\trim(\hat T_{\infty})\text{ contains at least }L(f)+1-\gamma\text{ repetitions}\right) \\
&\hspace{2cm}\leq 
\frac
{\Ec{\|\trim(\hat T_{\infty})\|^{2(L(f)+1-\gamma)}} 
+ 2\mathtt e\ \Ec{\|\trim(\hat T_{\infty})\|^{L(f)+1-\gamma}}}
{k^{L(f)+1-\gamma}},
\end{align*}
in view of Lemma~\ref{lem:trim},
which concludes the proof since $\binom k \gamma \geq (k-\gamma)^{\gamma}$ 
and all moments of $\|\trim(\hat T_{\infty})\|$ are finite.
\end{proof}

\subsection{Examples}
We show here how to apply Theorem~\ref{thm:theta} to different random trees: we first consider critical Galton--Watson trees, and then the Ford's alpha tree.

\subsubsection{Application to Galton--Watson trees (Proof of Theorem~\ref{prop:GW})}\label{sub:ex_GW}
For all $n\geq 1$, let $T_n$ be a critical Galton--Watson tree conditioned to have size $n$. 
Let us denote by $\xi$ its reproduction random variable: 
in particular, we have $\mathbb E \xi = 1$.
We also assume that there exists a positive constant $a$ such that $\mathbb E \mathtt e^{a\xi} <+\infty$
and that $\mathbb P(\xi=1)=0$.

It is known from the literature, and mentioned earlier in Section~\ref{sub:GW} that
$(T_n)_{n\geq 1}$ converges locally to an infinite random tree $T_{\infty}$ having a unique end,
on which are hanging some independent copies of the critical, unconditioned (and thus almost surely finite) Galton--Watson trees of reproduction $\xi$.
Thus, assumption~({\tt H}) is satisfied.

Moreover, the proof of Proposition~\ref{pro:bound-trimmed-GW} tells us that
the $k$-trimmed subtree $\tau=\trim(\hat T_{\infty})$ of $\hat T_{\infty}$ 
is such that there exists a constant $c>0$ such that,
for all $x>0$,
\[\mathbb P(\|\tau\|\geq x) \leq \exp\left(\nicefrac{-cx}{k^2}\right),\]
as long as there exists $a>0$ such that $\mathbb E \mathtt e^{a\xi} < \infty$, which is assumed here.

%The tree $\trim(\hat T_{\infty})$ is the tree obtained after 
%(1) cutting the spine a soon as this cut does not change the Boolean function calculated by $\hat T_{\infty}$,
%and (2) trimming each of the forests $A_i$ hanging on the spine.
%Thus, as already stated in Lemma~\ref{lem:simp<trim}, $\simp(\hat T_{\infty})$ is contained in $\tau$, and we have, 
%almost surely, $\|\tau\|\geq \|\simp(\hat T_{\infty})\|$.

Thus, for all $m\geq 1$, for all $x\geq 0$,
\[\mathbb P(\|\trim(\hat T_{\infty})\|^m \geq x)
\leq \exp\left(\frac{-c x^{\nicefrac1{m}}}{k^2}\right),\]
which implies that
\[\mathbb E\|\trim(\hat T_{\infty})\|^m < +\infty,\]
which proves that these Galton--Watson trees verify assumption $(ii)$ of Theorem~\ref{thm:loc-lim}.

Let us now check that Assumption $(i)$ is also verified:
By assumption, $\xi$ has exponential moments, implying that $\hat \xi$ also has exponential moments.
Note that, for all integers~$m$
\[\Delta_m:= N_{m+1}(T_{\infty}) - N_m(T_{\infty})
=  \sum_{i=1}^{\Delta_{m-1}- 1} \xi^{(m)}_i + \hat \xi^{(m)} + 1,\]
where the $\xi^{(m)}_i$'s are independent copies of $\xi$ and $\hat \xi^{(m)}$ is a copy of $\xi$.
Therefore, we can prove by induction on $m$ that, for all integers $m$, $N_m(T_{\infty})$ has exponential moments, 
which implies that Assumption $(i)$ holds.

Therefore, Theorem~\ref{thm:theta} applies to this class of Galton--Watson trees, 
including the Catalan tree model as a particular case, which proves Theorem~\ref{prop:GW}.

\subsubsection{Application to the alpha model (proof of Theorem~\ref{prop:alpha_tree})}\label{sub:ex_alpha}
We consider in this section the alpha--gamma model.
For technical reasons, we restrict ourselves to the case $\alpha=\gamma$,
although Theorem~\ref{prop:alpha_tree} might also hold for other values of $\gamma$.
When $\alpha=\gamma$, the probability that the random tree of size $n$ -- called the alpha tree of size $n$ -- 
has a left subtree of size $k\in\{1, \ldots, n-1\}$ is given by (see Equation~\eqref{eq:alpha-gamma} 
in which we take $\gamma=\alpha$)
\[q_n^{\alpha}(k)=
\frac{\Gamma(k-\alpha)\Gamma(n-k-\alpha)}{\Gamma(n-\alpha)\Gamma(1-\alpha)}
\left[\frac{\alpha}{2} \binom n k + (1-2\alpha)\binom{n-2}{k-1}\right].\]

The following lemmas aim at proving that the alpha model verifies Assumption~$(ii)$ of Theorem~\ref{thm:theta}.

% {\color{red}We consider the alpha-gamma model already described in Section~\ref{sec:examples}.
% It is a fragmentation process with the following splitting distribution:
% \begin{align*}
% & q_n^{\alpha,\gamma}(n_1,\dots, n_k) \\
% & = \frac{\Gamma(1-\alpha)}{\Gamma(n-\alpha)}
% \left(\gamma + \frac{1-\alpha-\gamma}{n(n-1)} \sum_{i\ne j} n_i n_j\right)
% \frac {\binom{n}{n_1,\dots, n_k} } {m_1! \dots m_n!} 
% \alpha^{k-2} \frac{\Gamma(k-1-\gamma/\alpha)}{\Gamma(1-\gamma/\alpha)}
% \prod_{j=1}^k \frac{\Gamma(n_j-\alpha)}{\Gamma(1-\alpha)},
% \end{align*}
% where $m_i=\#\{j: n_j=i\}$, for $1\le i\le n$.}

\begin{lem}\label{lem:proba_feuille}
For all $n\geq 1$, $q_n^{\alpha}(1) \geq \nicefrac{\alpha}{2}$.
%{\color{red}
%There exists $\eta >0$ such that, for all $n\geq 2$, the probability that root of the $\alpha, \gamma$ tree of size $n$ has at least one child of size 1 is lower bounded by $\eta$.}
\end{lem}

\begin{proof}
We have:
\[q_n^{\alpha}(1)
= \frac{\Gamma(n-1-\alpha)}{\Gamma(n)} \Big(\frac{\alpha n}{2}+(1-2\alpha)\Big) 
=\frac{\alpha}{2}\,\frac{n+\frac{2(1-2\alpha)}{\alpha}}{n-1-\alpha}\geq \frac{\alpha}{2},
\] for all integers $n$.
\end{proof}

As a preliminary, let us also prove that
\begin{lem}\label{lem:delta}
For all $\delta>0$, there exists $n_0\in\mathbb N$ (independent of $\delta$) such that, 
for all $n\geq n_0$,
\[\sum_{\delta n\le k \le (1-\delta)n} q_n^{\alpha}(k) \leq \frac C2\ \delta^{-2(\alpha+1)} n^{-\alpha},
\qquad \text{where}\qquad
C=2\ \frac{\max(\nicefrac{\alpha}{2},1-\nicefrac{3\alpha}{2})}{\Gamma(1-\alpha)}\,.
\]
%{\color{red}
%For all $\delta>0$, when $n\to+\infty$,
%\[\sum_{k\geq 2} k^{1-\gamma} \sum_{n_k \leq \ldots \leq n_1, n_2 \geq \delta n} 
%q_n^{\alpha, \gamma}(n_1, \ldots, n_k)
%=\mathcal O\left(\frac{\delta^{-2(\alpha+1)}}{1-2\delta} n^{-\gamma(1+\alpha-\gamma)}\right),\]
%where $\mathcal O$ constants do not depend on $\delta$.}
\end{lem}

\begin{proof}
For all $\delta>0$, asymptotically when $n\to+\infty$,
\begin{align*}
\sum_{\delta n \le k \le (1-\delta) n} q_n^{\alpha}(k)
&= \sum_{\delta n \le k \le (1-\delta) n} \frac{\Gamma(k-\alpha)\Gamma(n-k-\alpha)}{\Gamma(n-\alpha)\Gamma(1-\alpha)}
\left[\frac{\alpha}{2} \binom n k + (1-2\alpha)\binom{n-2}{k-1}\right]\\
&= \frac{\mathtt e^{\alpha+1} +o(1)}{\Gamma(1-\alpha)}\ 
\sum_{\delta n \le k \le (1-\delta) n}
\frac{(k-\alpha-1)^{k-\alpha-\nicefrac12}(n-k-\alpha-1)^{n-k-\alpha-\nicefrac12}}{(n-\alpha-1)^{n-\alpha-\nicefrac12}}\\
&\hspace{1cm}\left[\frac{\alpha}{2} \frac{n^{n+\nicefrac12}}{k^{k+\nicefrac12} (n-k)^{n-k+\nicefrac12}}
+ (1-2\alpha) \frac{(n-2)^{n-\nicefrac32}}{(k-1)^{k-\nicefrac12} (n-k-1)^{n-k-\nicefrac12}}\right]\\
&= \frac{\max(\nicefrac{\alpha}{2},1-\nicefrac{3\alpha}{2})+o(1)}{\Gamma(1-\alpha)}\ 
\sum_{\delta n \le k \le (1-\delta) n}
\left(\frac{n}{k(n-k)}\right)^{\alpha+1}\\
&
\leq \frac{C}{2} (1+o(1))\ n \left(\frac{n}{(\delta n)^2}\right)^{\alpha+1}
%= \frac{C}{1+\eta} \delta^{-2(\alpha+1)} n^{-\alpha}\  (1+o(1)),
\end{align*}
where $C=2\ \frac{\max(\nicefrac{\alpha}{2},1-\nicefrac{3\alpha}{2})}{\Gamma(1-\alpha)}$ does not depend on $\delta$.
\end{proof}

Let us define the following cutting procedure:
a node having at least one child of size 1 is cut with probability $\varepsilon>0$.
Let us denote by $\tau(n)$ the size of the alpha tree of size $n$ after having applied this cut procedure.
We are going to prove that all the moments of $\tau(n)$ are finite.

Let us start with its expectation:
\begin{lem}
There exists a constant $M_1$ such that, for all $n\geq 1$, 
$\mathbb E \tau(n) \leq \frac{M_1 n^{\alpha}}{1+\ln^2 n}$.
%{\color{red}
%For all $n\geq 1$, $\mathbb E \tau(n) \leq \frac{n^{\gamma}}{1+\ln n}$.}
\end{lem}

\begin{proof}
We prove this result by induction.

Assume that, for all $k\leq n-1$, $\mathbb E \tau(k) \leq \frac{M_1 k^{\alpha}}{1+\ln^2 k}$.
Using the induction hypothesis,
\begin{align*}
\mathbb E \tau(n)
&= \sum_{k=2}^{n-2} q_n^{\alpha}(k) \left(\mathbb E\tau(k)+ \mathbb E \tau(n-k)\right)
+ (1-\varepsilon) 2q_n^{\alpha}(1) \left(1+\mathbb E \tau(n-1)\right)\\
&\leq M_1 \sum_{k=2}^{n-2} q_n^{\alpha}(k) \left(\frac{k^{\alpha}}{1+\ln^2 k}+ \frac{(n-k)^{\alpha}}{1+\ln^2 (n-k)}\right)
+ (1-\varepsilon) 2q_n^{\alpha}(1) \left(1+\frac{M_1 (n-1)^{\alpha}}{1+\ln^2 (n-1)}\right)\\
&= M_1 \sum_{k=1}^{n-1} q_n^{\alpha}(k) \left(\frac{k^{\alpha}}{1+\ln^2 k}+ \frac{(n-k)^{\alpha}}{1+\ln^2 (n-k)}\right)
- 2 M_1 \varepsilon q_n^{\alpha}(1) \left(1+\frac{(n-1)^{\alpha}}{1+\ln^2 (n-1)}\right)\\
&\hspace{2cm}+ 2 (1-\varepsilon) q_n^{\alpha}(1) (1-M_1).
\end{align*}

First note that 
\[\left(1+\frac{(n-1)^{\alpha}}{1+\ln^2 (n-1)}\right)
\geq \frac{1+(n-1)^{\alpha}}{1+\ln^2 n}\geq \frac{n^{\alpha}}{1+\ln^2 n}.\]
Moreover, for all $\delta>0$, for {$C$ the constant of Lemma~\ref{lem:delta}} and all $n$ large enough
\begin{align*}
&\sum_{k=1}^{n-1} q_n^{\alpha}(k) \left(\frac{k^{\alpha}}{1+\ln^2 k}+ \frac{(n-k)^{\alpha}}{1+\ln^2 (n-k)}\right)\\
&\hspace{.5cm}= \sum_{\delta n \le k \le (1-\delta) n} q_n^{\alpha}(k) \left(\frac{k^{\alpha}}{1+\ln^2 k}+ \frac{(n-k)^{\alpha}}{1+\ln^2 (n-k)}\right)
+ \sum_{k<\delta n \text{ or }k>(1-\delta)n} q_n^{\alpha}(k) \left(\frac{k^{\alpha}}{1+\ln^2 k}+ \frac{(n-k)^{\alpha}}{1+\ln^2 (n-k)}\right)\\
&\hspace{.5cm}\leq 
\frac{2[(1-\delta)n]^{\alpha}}{1+\ln^2 (\delta n)}\ \sum_{\delta n \le k \le (1-\delta) n} q_n^{\alpha}(k)
+ \left(\frac{(\delta n)^{\alpha}}{1+\ln^2 [(1-\delta)n]} + \frac{n^{\alpha}}{1+\ln^2 [(1-\delta)n]}\right) 
\sum_{k<\delta n \text{ or }k>(1-\delta)n} q_n^{\alpha}(k)\\
&\hspace{.5cm}\leq 
\frac{2 C [(1-\delta)n]^{\alpha}}{1+\ln^2 (\delta n)}\ \delta^{-2(1+\alpha)} n^{-\alpha}
+ [\delta^{\alpha} + 1]\ \frac{1+\ln^2 n}{1+\ln^2 [(1-\delta)n]} \ \frac{n^{\alpha}}{1+\ln^2 n}\\
&\hspace{.5cm}\leq 
2 C (1-\delta)^{\alpha}\delta^{-2(1+\alpha)}
+ (\delta^{\alpha} + 1)\ \frac{1+\ln^2 n}{1+\ln^2 [(1-\delta)n]} \ \frac{n^{\alpha}}{1+\ln^2 n},
\end{align*}
where we have applied Lemma~\ref{lem:delta}.
In total, we thus get:
\begin{align*}
\mathbb E \tau(n)
\leq &\frac{M_1 n^{\alpha}}{1+\ln^2 n}
\left[(\delta^{\alpha}+1)\frac{1+\ln^2 n}{1+\ln^2 [(1-\delta)n]}- 2\varepsilon q_n^{\alpha}(1)\right]
+ M_1\left(2 C\delta^{-2(\alpha+1)}-2(1-\varepsilon)q_n^{\alpha}(1))\right)\\
&+ 2(1-\varepsilon)q_n^{\alpha}(1).
\end{align*}
Let us first fix $\delta$ such that $-2\eta:= \delta^{\alpha}-\varepsilon \alpha < 0$.
It implies that, in view of Lemma~\ref{lem:proba_feuille},
\[(\delta^{\alpha}+1)\frac{1+\ln^2 n}{1+\ln^2 [(1-\delta)n]}- 2\varepsilon q_n^{\alpha}(1)
\leq (\delta^{\alpha}+1)\frac{1+\ln^2 n}{1+\ln^2 [(1-\delta)n]}- 2\varepsilon \frac{\alpha}{2}
\to \delta^{\alpha}+1-2\varepsilon q_n^{\alpha}(1) = 1-2\eta,\]
when $n\to+\infty$. Thus, there exists $n_{\delta}$ such that, for all $n\geq n_{\delta}$,
\[(\delta^{\alpha}+1)\frac{1+\ln^2 n}{1+\ln^2 [(1-\delta)n]}- 2\varepsilon q_n^{\alpha}(1) < 1-\eta.\]
Thus, for all $n\geq n_{\delta}$, 
\[\mathbb E \tau(n)
\leq \frac{M_1 n^{\alpha}}{1+\ln^2 n} 
+ M_1\left(2C \delta^{-2(\alpha+1)}-2(1-\varepsilon)q_n^{\alpha}(1)-\frac{\eta n^{\alpha}}{1+\ln^2 n}\right)
+ 2(1-\varepsilon)q_n^{\alpha}(1).
\]
There exists $\tilde n_{\delta} > n_{\delta}$ such that, for all $n\geq \tilde n_{\delta}$,
\[2C \delta^{-2(\alpha+1)}-2(1-\varepsilon)q_n^{\alpha}(1)-\frac{\eta n^{\alpha}}{1+\ln^2 n} < -1,\]
thus, choosing
\[M_1:= 2 + \max_{n\leq \tilde n_{\delta}} \left\{\frac{1+\ln^2 n}{n^{\alpha}}\ \mathbb E \tau(n)\right\}\]
we have that, for all $n\geq 1$,
\[\mathbb E \tau(n)
\leq \frac{M_1 n^{\alpha}}{1+\ln^2 n}  - M_1 + 2(1-\varepsilon)
\leq \frac{M_1 n^{\alpha}}{1+\ln^2 n},\]
which concludes the proof.
\end{proof}

\begin{lem}\label{lem:expectation}
There exists a constant $K>0$ such that, for all $n\geq 2$, $\mathbb E \tau(n) \leq K$.
\end{lem}

\begin{proof}
We have proved in Section~\ref{sec:examples} that the $\alpha$ tree converges locally to an infinite spine on which are attached alpha trees of random sizes $(N_i)_{i\geq 0}$, from top to bottom in the tree. The sequence $(N_i)_{i\geq 0}$ is a sequence of i.i.d. random variables having law
\[\mathbb P(N_i = k) = \frac{\alpha\Gamma(k-\alpha)}{k!\Gamma(1-\alpha)}\text{ for all }k\geq 1.\]
Let us denote by $\tau(\infty)$ the size of this infinite tree after applying it the cutting procedure explained above.
First note that for all $i\geq 0$ $\mathbb P(N_i=1) = \alpha$.
Thus, if we explore the spine from top to bottom, we cut the spine between $u_i$ and $u_{i+1}$ with probability at least $\alpha \varepsilon$ independently for all $i\geq 0$.
Thus, the spine is cut at height at most $H$ where $H$ is a geometric random variable with parameter $\varepsilon \alpha$.
Thus,
\[
\mathbb E\tau(\infty) 
= \mathbb E H \mathbb E \tau(N_1)
= \frac1{\alpha \varepsilon} \sum_{k\geq 1}\frac{\alpha\Gamma(k-\alpha)}{k!\Gamma(1-\alpha)} \mathbb E \tau(k)
\leq \frac1{\alpha \varepsilon} \sum_{k\geq 1}\frac{\alpha\Gamma(k-\alpha)}{k!\Gamma(1-\alpha)} \frac{k^{\alpha}}{1+\ln^2 k}.
\]
Note that
\[\frac{\alpha\Gamma(k-\alpha)}{k!\Gamma(1-\alpha)} \frac{k^{\alpha}}{1+\ln^2 k}
\sim \frac{\alpha}{\Gamma(1-\alpha)} \frac{1}{k\ln^2 k},\]
when $k\to+\infty$. Thus, it is summable, implying that
$\mathbb E\tau(\infty) < +\infty$.
Finally remark that, for all $n\geq 2$,
\[\mathbb E \tau(n) \leq \mathbb E\tau(\infty)=:K,\]
which concludes the proof.
\end{proof}

\begin{lem}\label{lem:moments}
For all $p\geq 1$, there exists a constant $K_p$ such that, for all $n\geq 2$,
$\mathbb E\tau(n)^p \leq K_p$.
\end{lem}

\begin{proof}
Lemma~\ref{lem:expectation} is the case $p=1$.
For all $p\geq 1$, $\mathbb E \tau(1)^p = 1$.

Let us fix $p$, and $n$, and assume that, 
\begin{itemize}
\item for all $q\leq p$, for all $n\geq 2$, $\mathbb E \tau(n)^q \leq K_q$; and that,
\item Let 
\[M_p:= \frac{\sum_{i=1}^{p-1} \binom p i K_i K_{p-i}}{2\varepsilon q_n^{\alpha}(1)},\] 
and assume that, for all $k\leq n$, $\mathbb E \tau(k)^p \leq \frac{M_p k^\alpha}{1+\ln^2 k}$.
\end{itemize}
Applying the induction hypothesis, we have:
\begin{align*}
\mathbb E \tau(n)^p
=& \sum_{k=2}^{n-2} q_n^{\alpha}(k) \mathbb E[(\tau(k)+\tau(n-k))^p]
+ (1-\varepsilon) 2 q_n^{\alpha}(1) \mathbb E[(1+\tau(n-1))^p]\\
=& \sum_{k=2}^{n-2} q_n^{\alpha}(k) \sum_{i=0}^p \binom p i \mathbb E \tau(k)^i \mathbb E \tau(n-k)^(p-i)
+ (1-\varepsilon) 2 q_n^{\alpha}(1) \sum_{i=0}^p \binom p i \mathbb E[\tau(n-1)^i]\\
\leq & \sum_{k=2}^{n-2} q_n^{\alpha}(k) \sum_{i=1}^{p-1} \binom p i K_i K_{p-i} + 
M_p \sum_{k=2}^{n-2} q_n^{\alpha}(k) \left(\frac{k^{\alpha}}{1+\ln^2 k} + \frac{(n-k)^{\alpha}}{1+\ln^2 (n-k)}\right)\\
&+ (1-\varepsilon) 2 q_n^{\alpha}(1) \sum_{i=1}^{p-1} \binom p i K_i 
+ M_p (1-\varepsilon) 2 q_n^{\alpha}(1) \left(1+ \frac{(n-1)^{\alpha}}{1+\ln^2 (n-1)}\right)\\
=&\sum_{i=1}^{p-1} \binom p i K_i K_{p-i} - 2\varepsilon q_n^{\alpha}(1) \sum_{i=1}^{p-1} \binom p i K_i \\
&+ M_p\left[\sum_{k=1}^{n-1} q_n^{\alpha}(k) \left(\frac{k^{\alpha}}{1+\ln^2 k} + \frac{(n-k)^{\alpha}}{1+\ln^2 (n-k)}\right)
- 2\varepsilon q_n^{\alpha}(1) \left(1+ \frac{(n-1)^{\alpha}}{1+\ln^2 (n-1)}\right)\right]
\end{align*}
Reusing the calculations made in the proof of Lemma~\ref{lem:expectation}, we get that:
\begin{align*}
\mathbb E \tau(n)^p
&\leq \sum_{i=1}^{p-1} \binom p i K_i K_{p-i} + \left(1-2\varepsilon q_n^{\alpha}(1)\right) \frac{M_p n^{\alpha}}{1+\ln^2 n}\\
&\leq\frac{M_p n^{\alpha}}{1+\ln^2 n} + \sum_{i=1}^{p-1} \binom p i K_i K_{p-i} 
- 2\varepsilon q_n^{\alpha}(1) M_p\\
&\leq \frac{M_p n^{\alpha}}{1+\ln^2 n},
\end{align*}
in view of the value of $M_p$.

We have thus proved that there exists a constant $M_p$ such that, for all $n\geq 2$, $\mathbb E \tau(n)^p\leq \frac{M_p n^{\alpha}}{1+\ln^2 n}$.
To prove that this implies that $\mathbb E \tau(n)^p <+\infty$, let us recall that the local limit of the alpha-tree is an infinite spine on which are hung some independent alpha-trees of respective sizes $(N_i)_{i\geq 1}$. We have already noticed that if we apply the cutting procedure to this infinite tree, the spine is cut at level $H$, where $H$ is a geometric random variable of parameter $\alpha \varepsilon$. 
Note that
\[\mathbb E \tau(N_1)^p 
= \sum_{k\geq 1} \frac{\alpha\Gamma(k-\alpha)}{k!\Gamma(1-\alpha)} \mathbb E \tau(k)^p
\leq \frac{M_p\alpha}{\Gamma(1-\alpha)} \sum_{k\geq 1} \frac{\Gamma(k-\alpha)}{k!} \frac{k^{\alpha}}{1+\ln^2 k} < +\infty.\]
If we denote by $\tau(\infty)$ the size of the obtained tree, then
\[\mathbb E \tau(\infty)^p 
= \mathbb E \left(\sum_{i=1}^H \tau(N_i)\right)^p
\leq \mathbb E \left[H^{p-1} \sum_{i=1}^H \tau(N_i)^p\right]
= \mathbb E H^p \mathbb E \tau(N_1)^p =: K_p < +\infty,
\]
which concludes the proof because $\mathbb E \tau(n)^p \leq \mathbb E \tau(\infty)^p$ for all $n\geq 2$.
\end{proof}

We are now ready to prove that the alpha-tree verifies the hypothesis of Theorem~\ref{thm:theta}, 
by relating the idealized cutting procedure described above to the trimming procedure.
We have already proved (see Section~\ref{sub:alpha-gamma}) that the alpha tree locally converges 
to an infinite tree $T_{\infty}$ which consists of an infinite spine to which are attached independent 
alpha trees of random i.i.d. sizes $(N_i)_{i\geq 0}$ where 
\[\mathbb P (N_1= k) = \frac{\alpha\Gamma(k-\alpha)}{k!\Gamma(1-\alpha)}.\]
Hypothesis ({\tt H}) is thus verified.

The alpha tree is by construction binary (see Figure~\ref{fig:alpha_tree}), and thus, Assumption $(i)$ of Theorem~\ref{thm:theta}
is trivially verified: for all integer $m\geq 1$, $N^{m+2}_m(T_{\infty}) \leq 2^{m(m+2)}<+\infty$. 

Let us couple the trimming procedure applied on $\hat T_{\infty}$ to the following procedure. 
Let $\varepsilon = \nicefrac1{2k}$. 
At first, all nodes of the local limit of the alpha tree $T_{\infty}$ are black: starting from the root,
\begin{itemize}
\item if a black node has a child of size $1$, then colour it and all its descendent in red; otherwise, colour it in blue;
\item when all nodes are coloured either in blue or in red, apply the cutting procedure above for all red parts of the tree: a red node having a child of size one is cut with probability $\varepsilon$. 
\end{itemize}
One can couple the above procedure with $\trim$ so that the size of the obtained tree by 
the procedure above is larger than the size of $\trim(\hat T_{\infty})$.
The black nodes correspond to the nodes in $\hat T_{\infty}$ having empty constraints sets,
the red nodes correspond to nodes having a non-empty constraint set: 
thus each leaf can make all its siblings inconsistent with probability at least $\nicefrac1{2k}$.
Since the internal nodes that were parent of a set of inconsistent siblings do not contribute to the size of $\trim(\hat T_{\infty})$,
cutting them together with their children does not affect the size of $\trim(\hat T_{\infty})$ 
(but it might change the Boolean function calculated by the tree).
Thus, the above procedure gives a random tree $\tau$ whose size stochastically dominates $\trim(\hat T_{\infty})$.

Moreover this tree $\tau$ consists in a blue tree to whose leaves are attached some red trees.
Let us denote by $\|\mathtt{blue}\|$ the number of the blue nodes in the above tree: 
Lemma~\ref{lem:moments} applied in the special case with $\varepsilon = 1$ gives us that all the moments of $\|\mathtt{blue}\|$ are finite.
Then, the red trees attached to the leaves of the blue tree are i.i.d. There is at most $\|\mathtt{blue}\|$ of them and by Lemma~\ref{lem:moments} applied to $\varepsilon = \frac1{2k}$, their sizes all have finite moments.

Therefore, the size of the whole tree (blue nodes and red nodes) has finite moments, which permits to conclude that $\|\trim(\hat T_{\infty})\|$ has finite moments. We can then apply our main result Theorem~\ref{thm:theta} to the alpha model and deduce Theorem~\ref{prop:alpha_tree}.

\setlength{\bibsep}{0.3em}
\bibliographystyle{plainnat}
\bibliography{andor}

\end{document}